\documentclass[a4paper,final]{amsart}

\usepackage{epic,eepic}
\usepackage{graphicx,color}
\usepackage[color,notref,notcite]{showkeys}
\usepackage{graphicx}
\usepackage{url}
\usepackage{amssymb,amsmath}
\usepackage{colortbl}
\usepackage{multicol}
\usepackage{multirow}
\usepackage{bm,cite}
\usepackage{amsaddr}

\definecolor{labelkey}{rgb}{0.6,0,1}

\usepackage[normalem]{ulem}
\definecolor{violet}{rgb}{0.580,0.,0.827}

\newcounter{cst}
\catcode`\@=11
\def\ctel#1{C_{\refstepcounter{cst}\@bsphack
\protected@write\@auxout{}%
           {\string\newlabel{#1}{{\thecst}{\thepage}}}\thecst}}
\catcode`\@=12

\newcommand{\cter}[1]{C_{\ref{#1}}}

\newcounter{cexp}
\catcode`\@=11
\def\terml#1{T_{\refstepcounter{cexp}\@bsphack
\protected@write\@auxout{}%
           {\string\newlabel{#1}{{\thecexp}{\thepage}}}\thecexp}}
\catcode`\@=12

\newcommand{\mathbi}[1]{{\boldsymbol #1}}

\renewenvironment{proof}{\noindent {\bf Proof } }{$\square$ }

\newtheorem{theorem}{Theorem}[section]
\newtheorem{remark}[theorem]{Remark}

\newtheorem{lemma}[theorem]{Lemma} 
\newtheorem{definition}[theorem]{Definition}

\newtheorem{corollary}[theorem]{Corollary}

\numberwithin{equation}{section}

\newcommand{\ba}{\begin{array}{llll}   }
\newcommand{\bac}{\begin{array}{c}}
\newcommand{\bari}{\begin{array}{r}}
\newcommand{\ea}{\end{array}}

\newcommand{\ban}{\begin{array}{llll}}
\newcommand{\ean}{\end{array}}

\newcommand{\be}{\begin{equation}}
\newcommand{\ee}{\end{equation}}

\newcommand{\beqsys }{\beqtab \left \{ \begin{array}{l}}
\newcommand{\eeqsys }{\end{array} \right . \eeqtab }

\newcommand{\benum}{\begin{enumerate}}
\newcommand{\eenum}{\end{enumerate}}

\newcommand{\beqtab}{\begin{eqnarray}} 
\newcommand{\eeqtab}{\end{eqnarray}}

\newcommand{\dsp}{\displaystyle}

\def\H{\mathbf{H}}
\def\L{\mathbf{L}}

\def\Gd{{\Gamma_D}}
\def\Gn{{\Gamma_N}}

\def\X#1{\mathbf{X}_{#1,\Gd}}

\def\XD#1{\mathbf{X}_{#1,\Gd}}

\newcommand{\stiff}{\mathbb{C}}
\newcommand{\stab}{\mathbb{D}}
\newcommand{\strain}{\mbox{\boldmath{$\varepsilon$}}}
\newcommand{\stress}{\mbox{\boldmath{$\sigma$}}}

\def\bu{\mathbf{u}}
\def\bv{\mathbf{v}}
\def\bw{\mathbf{w}}
\def\buex{\overline{\mathbf{u}}}

\def\bvarphi{\mathbi{\varphi}}
\def\bF{\mathbf{F}}
\def\bD{\mathbf{\Delta}}
\def\bg{\mathbf{g}}

\def\bh{\mathbf{h}}
\def\sym{\mathcal S_{d\times d}}

\def\T{{\mathcal T}}
\def\bV{\mathbf{V}}
\def\CT{\mathcal{T}}
\def\n{\mathbf{n}}
\def\x{\mathbi{x}}

\renewcommand{\d}{{\rm d}}

\def\argmin{\mathop{\rm argmin}}

\newcommand{\disc}{{\mathcal{D}}}

\renewcommand{\div}{{\rm div}}

\renewcommand{\O}{\Omega}
\newcommand{\N}{\mathbb N}
\newcommand{\R}{\mathbf{R}}
\def\dev{\mathop{\rm dev}}
\def\tr{\mathop{\rm tr}}

\newcommand{\btau}{\mbox{\boldmath{$\tau$}}}
\newcommand{\bomega}{\mbox{\boldmath{$\omega$}}}
\newcommand{\bd}{\mbox{\boldmath{$d$}}}
\def\bL{\btau}
\def\bG{\bomega}

\begin{document}

\title{Gradient Schemes for linear and non-linear elasticity equations}

\author{J\'er\^ome Droniou}
\address{School of Mathematical Sciences,
Monash University, Victoria 3800, Australia.}

\author{Bishnu P. Lamichhane}
\address{School of Mathematical \& Physical Sciences, University of Newcastle,
University Drive, NSW 2308, Callaghan, Australia.}

\email{jerome.droniou@monash.edu,Bishnu.Lamichhane@newcastle.edu.au}

\date{19 November 2013}

\begin{abstract}
The Gradient Scheme framework provides a unified analysis setting for
many different families of numerical methods for diffusion equations.
We show in this paper that the Gradient Scheme framework can be adapted to elasticity equations,
and provides error estimates for linear elasticity and convergence
results for non-linear elasticity. We also establish that several
classical and modern numerical methods for elasticity are embedded in the
Gradient Scheme framework, which allows us to obtain convergence results
for these methods in cases where the solution does not satisfy the
full $H^2$-regularity or for non-linear models.
\end{abstract}



\maketitle

\section{Introduction}\label{sec:intro}

We are interested in the numerical approximation of
the (possibly non-linear) elasticity equation
\be\label{eq:elas-strong}
\ba
-\div(\stress(x,\strain(\buex))=\bF\,,&\mbox{in }\Omega,\\
\strain(\buex)=
\frac{\nabla \buex + (\nabla \buex)^T}{2}\,,&\mbox{in }\Omega,\\
\buex = 0\,,&\mbox{on }\Gd\,,\\
\stress(x,\strain(\buex))\n=\bg\,,&\mbox{on }\Gn,
\ea
\ee
where $\Omega\subset \R^d$ is the body submitted to the force field $\bF$, 
$\mathbf{n}$ is the unit normal to $\partial\Omega$ pointing outward $\Omega$,
$\Gd$ and $\Gn$ are subsets of $\partial\Omega$ on which the body is
respectively fixed and submitted to traction, $\stress$ and $\strain$
are the second-order stress and strain tensors, $\buex=(\bar u_i)_{i=1,\ldots,d}:\Omega\to\R^d$ describes
local displacements and the gradient is written in columns: $\nabla\buex=
(\partial_j \bar u_i)_{i,j=1,\ldots,d}$.

This formulation of elasticity equations covers a number of classical models:
\begin{itemize}
\item the linear elasticity model with $\stress(x,\strain(\bu))=\stiff(x)\strain(\bu)$,
in which $\stiff$ is a $4$th order stiffness tensor,
\item the damage models of \cite{CCC10b} with $\stress(x,\strain(\bu))=(1-D(\strain(\bu)))
\stiff(x)\strain(\bu)$, where the damage index $D$ is a scalar function,
\item the non-linear Hencky-von Mises elasticity model \cite{NEC86}
in which $\stress(x,\strain(\bu))=\widetilde{\lambda}(\dev(\strain(\bu)))\tr(\strain(\bu))\mathbf{I}
+2\widetilde{\mu}(\dev(\strain(\bu)))\strain(\bu)$, where $\widetilde{\lambda}$ and
$\widetilde{\mu}$ are the non-linear Lam\'e coefficients, $\tr$ is the trace
operator and $\dev(\bL)=(\bL-\frac{1}{2}\tr(\bL)\mathbf{I}):
(\bL-\frac{1}{2}\tr(\bL)\mathbf{I})$ is the deviatoric operator.
\end{itemize}

Convergence of conforming Finite Element methods for the linear elasticity problem can be 
obtained by using standard techniques \cite{Cia78,BS94}. This convergence analysis 
covers the case when the solution does not possess a full $H^2$-regularity. 
However, convergence analysis for non-conforming Finite Element methods 
is most often done using the full $H^2$-regularity of the solution 
\cite{BS92,BS94,BCR04,LRW06,BH06}. Similarly, the
convergence of numerical methods for non-linear elasticity models
only seems to have been established for conforming approximations (i.e. the space(s) of approximate solutions
are subspaces of the space(s) of continuous solutions, whether a displacement
or several-fields formulation is chosen) and assuming
the full $H^2$-regularity of the solution \cite{GS02,CD04,BM05}.

The Gradient Scheme framework is a setting, based on a few discrete
elements and properties, which has been recently developed to analyse
numerical methods for a vast number of diffusion models: 
linear or non-linear, local or non-local, 
stationary or transient models, etc. (see \cite{EYM11,DRO12,EYM13-2,EYM11-2,DRO13}).
This framework is also currently being extended to the linear poroelasticity equation, see
\cite{AEL13}.
It has been shown that a number of well-known
methods for diffusion equations are Gradient Schemes \cite{EYM12,DRO12,EYM13,DRO13}:
Galerkin methods (including conforming Finite Element methods), Mixed Finite Element methods,
Hybrid Mimetic Mixed methods (including Hybrid Finite Volumes,
Mimetic Finite Differences and Mixed Finite Volumes), Discrete Duality
Finite Volume methods, etc.
Moreover, the Gradient Scheme framework enables convergence analysis
of all these numerical methods for all the afore-mentioned models
under very unrestrictive assumptions. 
The key feature of Gradient Schemes that they
provide a unified framework for the convergence analysis
of many different numerical schemes for linear and non-linear diffusion
equations without assuming the full $H^2$-regularity of the solution.
In practice, the full $H^2$-regularity is not achieved due to 
the non-convexity of the domain,  corner singularities, discontinuities of the stiffness tensor, 
non-smooth data and mixed boundary conditions.

The aim of this paper is to extend the Gradient Scheme framework to linear and non-linear
elasticity models, thus showing that all the advantages of this analysis framework
can be applied to classical numerical techniques developed for elasticity equations.
The paper is organised as follows. In the next section, we introduce
the notion of Gradient Discretisations, used to define Gradient Schemes
for \eqref{eq:elas-strong}. We also state the three properties,
\emph{consistency}, \emph{limit-conformity} and \emph{coercivity},
that a Gradient Discretisation must satisfy in order to lead to a stable
and convergent numerical scheme. In Section \ref{sec:lin}, we first analyse the
convergence of Gradient Schemes for linear elasticity equations, providing
an error estimate under very weak regularity assumptions on the
data and solution.
We then carry out the convergence analysis
for fully non-linear models, proving the convergence of the approximate
solution under the same unrestrictive assumptions. Section \ref{sec:ex}
is devoted to the study of some examples of Gradient Scheme.
We show in particular that many schemes for elasticity equations,
including methods developed to handle the nearly incompressible limit
and acute bending, do fall in the framework of Gradient Schemes and that
our convergence analysis -- for both linear and non-linear models -- therefore
applies to them. Some conclusions of the paper are summarised in the final section.

\section{Definition of Gradient Schemes for the elasticity equation}\label{sec:gs}

Our general assumptions on the data are as follows.
\be\label{hyp:omega}
\ba
\Omega\mbox{ is a connected open subset of $\R^d$ ($d\ge 1$) with Lipschitz boundary,}\\
\mbox{$\Gd$ and $\Gn$ are disjoint subsets of $\partial\Omega$ such that $\partial\Omega=
\Gd\cup\Gn$ and}\\
\mbox{$\Gd$ has a non-zero $(d-1)$-dimensional measure},
\ea
\ee
\be\label{hyp:fg}
\bF\in \L^2(\Omega)\,,\quad \bg\in \L^2(\Gn)
\ee
(where $\L^2(X)=(L^2(X))^d$) and, denoting by $\sym$ the set of symmetric $d\times d$ tensors,
\be\label{hyp:stress}
\ba
\dsp \stress:(x,\bL)\in\Omega\times \sym
\mapsto \stress(x,\bL)\in \sym\mbox{ is a Caratheodory}\\
\dsp\mbox{function (i.e. measurable w.r.t. $x$ and continuous w.r.t. $\bL$) and}\\
\exists \sigma^*,\sigma_*>0\mbox{ such that, for a.e. }x\in\Omega\,,\;
\forall \bL,\bG\in \sym\,,\\
\quad\begin{array}{l@{\qquad}l}
|\stress(x,\bL)|\le \sigma^*|\bL|+\sigma^*&\mbox{(growth)},\\
\stress(x,\bL):\bL\ge \sigma_*|\bL|^2&\mbox{(coercivity)},\\
(\stress(x,\bL)-\stress(x,\bG)):(\bL-\bG)\ge 0&\mbox{(monotonicity)},
\end{array}
\ea
\ee
where, for $\bL,\bG\in \R^{d\times d}$,
$\bL:\bG=\sum_{i,j=1}^d \bL_{ij}\bG_{ij}$ and $|\bL|^2=\bL:\bL$. 
In the following, we also denote by
$\cdot$ and $|\cdot|$ the Euclidean product and norm on $\R^d$. 

\begin{remark} Note that the linear elasticity and the Hencky-von Mises models both
satisfy these assumptions (see \cite[Lemma 4.1]{BAR02}
for a proof of the monotonicity of the Hencky-von Mises model).
One can also see that the damage model $\stress(x,\strain(\bu))=(1-D(\strain(\bu)))\stiff(x)\strain(\bu)$
satisfies \eqref{hyp:stress} if $1-D(\xi)=f(|\xi|)$ where, for some
$0<\underline{d}\le \overline{d}$, $f$ is continuous
$[0,\infty)\to [\underline{d},\overline{d}]$ and such that
$s\in[0,\infty)\to sf(s)$ is non-decreasing.
\label{rem:dam}\end{remark}

Under these assumptions, and defining $\H^1(\Omega)=H^1(\Omega)^d$,
$\gamma:\H^1(\Omega)\to \L^{2}(\partial\Omega)$ the trace operator and $\H^1_{\Gd}(\Omega)=
\{\bv\in \H^1(\Omega)\,:\,\gamma(\bv)=0\mbox{ on $\Gd$}\}$,
the weak formulation of \eqref{eq:elas-strong} is
\be\label{eq:elas-weak}
\ba
\dsp\mbox{Find $\buex\in \H^1_{\Gd}(\Omega)$ such that, for any $\bv
\in \H^1_{\Gd}(\Omega)$},\\
\ba\dsp \int_\Omega \stress(x,\strain(\buex)(x)):\strain(\bv)(x)\d x&=&\dsp 
\int_\Omega \bF(x)\cdot\bv(x)\d x\\
&&\dsp +\int_{\Gn}\bg(x)\cdot\gamma(\bv)(x)\d S(x).
\ea
\ea
\ee

Gradient Schemes for such equations are based on Gradient Discretisations,
which consist in introducing a discrete space, gradient, trace and reconstructed function,
and using those to approximate \eqref{eq:elas-weak}.
The following definitions are adapted to elasticity equations, and to
non-homogeneous mixed boundary conditions, from the theory developed in \cite{EYM11,DRO12}
for diffusion equations with homogeneous Dirichlet boundary conditions.

\begin{definition}[Gradient Discretisation for the elasticity equation]
\label{def:grad-disc}~\\
A Gradient Discretisation $\disc$ for Problem \eqref{eq:elas-strong}
is $\disc =(\X{\disc},\Pi_\disc,\T_{\disc},\nabla_\disc)$, where:
\begin{enumerate}
\item the set of discrete unknowns $\X{\disc}$ is a finite dimensional
vector space on $\R$ whose definition includes the null trace condition
on $\Gd$,
\item the linear mapping $\Pi_\disc~:~\X{\disc}\to \L^2(\O)$ is the reconstruction of the approxi\-ma\-te function,
\item the linear mapping $\T_{\disc}:\X{\disc}\to
\L^2(\Gn)$ is a discrete trace operator,
\item the linear mapping $\nabla_\disc~:~\X{\disc}\to \L^2(\O)^{d}$ is the discrete
gradient operator. It must be chosen such that $\Vert \cdot \Vert_{\disc} := \Vert \nabla_\disc \cdot \Vert_{\L^2(\O)^{d}}$ is a norm on $\X{\disc}$.
\end{enumerate}
\end{definition}

Once a Gradient Discretisation is available, the related Gradient Scheme consists in
writing the weak formulation \eqref{eq:elas-weak} with the continuous spaces and
operators replaced by their discrete counterparts.

\begin{definition}[Gradient Scheme for the elasticity equation]\label{def:grad-scheme}~\\
If $\disc = (\X{\disc},\Pi_\disc,\T_{\disc},\nabla_\disc)$ is a Gradient Discretisation
in the sense of Definition \ref{def:grad-disc}
then we define the related \emph{Gradient Scheme} for \eqref{eq:elas-strong} by
\be\label{grad-scheme}
\begin{array}{l}
\dsp \mbox{Find $\bu\in \X{\disc}$ such that,
$\forall \bv\in \X{\disc}$,}\\[0.5em]
\ba
\dsp \int_\Omega \stress(x,\strain_\disc(\bu)(x)):\strain_\disc(\bv)(x)\d x&=&\dsp
\int_\Omega \bF(x)\cdot\Pi_\disc\bv(x)\d x\\
&&\dsp\qquad+\int_{\Gn}\bg(x)\cdot\T_\disc(\bv)(x)\d S(x)
\ea
\end{array}\ee
where $\strain_\disc(\bv)=\frac{\nabla_\disc \bv+(\nabla_\disc \bv)^T}{2}$.
\end{definition}

The definitions of consistency, limit-conformity and compactness
of Gradient Discretisations for Equation \eqref{eq:elas-strong} are the same as for diffusion equations,
taking into account the fact that functions are vector- or tensor-valued
in the elasticity model.

The consistency of a sequence of Gradient Discretisations ensure that
any function in the energy space can be approximated, along with its gradient,
by discrete functions.

\begin{definition}[Consistency] \label{def:cons}
Let $\disc$ be a Gradient Discretisation in the sense of
Definition \ref{def:grad-disc}, and let $S_{\disc}:\H^1_{\Gd}(\Omega)\to [0,+\infty)$
be defined by
\begin{equation}
\ba
\dsp \forall \bvarphi\in \H^1_{\Gd}(\Omega)\,,\\[0.5em]
\dsp S_{\disc}(\bvarphi) = \min_{\bv\in \X{\disc}}\big\{\Vert \Pi_\disc \bv - \bvarphi\Vert_{\L^2(\O)} + \Vert 
\nabla_\disc \bv -\nabla\bvarphi\Vert_{\L^2(\O)^{d}}\big\}.
\ea
\label{def:sdisc}
\end{equation}
A sequence $(\disc_m)_{m\in\N}$ of Gradient Discretisations is said to be \textbf{consistent} 
if,  for all $\bvarphi\in \H^1_{\Gd}(\Omega)$, $S_{\disc_m}(\bvarphi)\to 0$ as  $m\to\infty$.
\end{definition}

The limit-conformity of a sequence of Gradient Discretisations ensures that the
dual of the discrete gradient behaves as an approximation of the divergence operator.
We let
\[
\H_{\div}(\Omega,\Gn) = \{\bL\in \L^2(\Omega)^{d}\,:\,
\div\bL\in \L^{2}(\O)\,,\;\gamma_\n(\bL)\in\L^2(\Gn)\}
\]
where $\gamma_\n(\bL)$ is the normal trace of $\bL$. This normal trace is well defined
in $\H^{-1/2}(\partial\Omega)$ if $\bL\in \L^2(\Omega)^d$ and $\div(\bL)\in \L^2(\Omega)$
(\footnote{The divergence of a tensor $\bL$ is taken row by row,
i.e. if $\bL=(\bL_{i,j})_{i,j=1,\ldots,d}$ then
$\div(\bL)=(\sum_{j=1}^d\partial_j \bL_{i,j})_{i=1,\ldots,d}$. This definition is
consistent with our definition of $\nabla$ by column in the sense that $-\div$ is
the formal dual operator of $\nabla$.}).

\begin{definition}[Limit-conformity] \label{def:lim-conf}
Let $\disc$ be a Gradient Discretisation
in the sense of Definition \ref{def:grad-disc}.
We define $W_{\disc}$: $\H_{\div}(\Omega,\Gn)^d\to [0,+\infty)$  by
\begin{equation}
\begin{array}{l@{}l}
\dsp \forall \bL\in \H_{\div}(\Omega,\Gn)^d\,,\\
\dsp W_{\disc}(\bL) = \mathop{\max_{\bv\in \X{\disc}}}_{\bv\not=0}
\frac{1}{\Vert  \bv \Vert_{\disc}}&\dsp \Bigg|
\int_\Omega \big(\nabla_\disc \bv(x):\bL(x) +  \Pi_\disc \bv(x) \cdot \div(\bL)(x)\big)  \d x\\ 
&\dsp \quad-\int_{\Gn}\gamma_\n(\bL)(x)\cdot\T_\disc(\bv)(x)\d S(x)\Bigg|.
\end{array}
\label{def:wdisc}\end{equation}
A sequence $(\disc_m)_{m\in\N}$ of Gradient Discretisations is said to be
\textbf{limit-conforming} if, for all $\bL\in \H_{\div}(\Omega,\Gn)^d$,
$W_{\disc_m}(\bL)\to 0$ as $m\to\infty$.
\end{definition}

The definition of coercivity of Gradient Discretisations for the elasticity
equation starts in the same way as for diffusion equations.
However, since the natural energy estimate for
elasticity equations is not on $\nabla\buex$ but on $\strain(\buex)$,
as in the continuous case we must add to it some discrete form of
K\"orn's inequality.

\begin{definition}[Coercivity] \label{def:coer}
Let $\disc$ be a Gradient Discretisation
in the sense of Definition \ref{def:grad-disc}. We define $C_\disc$
(maximum of the norms of the linear mappings $\Pi_\disc$
and $\T_\disc$) by
\be
C_\disc =  \max_{\bv\in \X{\disc}\setminus\{0\}}\left(\frac {\Vert \Pi_\disc \bv\Vert_{\L^2(\O)}} {\Vert \bv \Vert_{\disc}},\frac {\Vert \T_\disc \bv\Vert_{\L^2(\Gn)}} {\Vert \bv \Vert_{\disc}}\right).
\label{def:Cdisc}\ee 
and $K_\disc$ (constant of the discrete K\"orn inequality) by
\be
K_\disc =  \max_{\bv\in \X{\disc}\setminus\{0\}}\frac {\Vert \bv\Vert_{\disc}} {\Vert \strain_\disc(\bv) \Vert_{\L^2(\O)^{d}}}.
\label{def:Kdisc}\ee 

A sequence $(\disc_m)_{m\in\N}$ of Gradient Discretisations is said to be \textbf{coercive} if  there exists $C_P>0$ such that $C_{\disc_m}+K_{\disc_m} \le C_P$ for all $m\in\N$.
\end{definition}

The definition of $C_\disc$ gives the following discrete Poincar\'e's inequality:
\be\label{ineq:poincare}
\forall \bv\in \X{\disc}\,:\quad
\Vert \Pi_\disc \bv\Vert_{\L^2(\O)} \le C_\disc \Vert \nabla_\disc \bv \Vert_{\L^2(\O)^{d}}.
\ee

\begin{remark}[Non-homogeneous Dirichlet boundary conditions]
Non-homogene\-ous Dirichlet boundary conditions $\buex=\bh$
can also be considered in \eqref{eq:elas-strong} and in the framework of
Gradient Schemes, upon introducing an interpolation operator and
modifying the definition of limit-conformity to take into account this interpolation
operator. See \cite{DRO13} for diffusion equations.
\end{remark}

\begin{remark} Although it does not seem to relate to any elasticity
model we know of, we could also handle a dependency of $\stress$ on $\buex$, i.e.
$\stress=\stress(x,\buex,\strain(\buex))$, upon adding
a compactness property of Gradient Discretisations (see \cite{DRO12}
for the handling of such lower order terms in diffusion equations).
\end{remark}

\section{Convergence analysis}

\subsection{Linear case}\label{sec:lin}

We assume here that the relationship between the strain and stress is linear,
and thus given by a stiffness $4$th order tensor:

\be
\ba
\mbox{There exists a measurable $\stiff:\Omega\to \R^{d^4}$ such that
$\stress(x,\bL)=\stiff(x)\bL$ and}\\
\exists \sigma^*,\sigma_*>0\mbox{ s.t., for a.e. }x\in\Omega\,,\;\forall \bL,\bG\in \R^{d\times d}\,,\\
\begin{array}{ll}
\stiff(x)\bL:\bG = \bL : \stiff(x)\bG\mbox{ and }
(\stiff(x)\bL)^T=\stiff(x)\bL^T\,, &\mbox{(symmetry)},\\
|\stiff(x)|\le \sigma^*&\mbox{(bound)},\\
\stiff(x)\bL:\bL\ge \sigma_*|\bL|^2&\mbox{(coercivity)},
\end{array}
\ea
\label{hyp:stiff}\ee

\begin{remark} These assumptions imply \eqref{hyp:stress} and cover
the classical linear elasticity model $\stress(x,\strain(\bu))=\lambda \tr(\strain(\bu))\mathbf{I}
+2\mu\strain(\bu)$ (i.e. the Hencky-von Mises model
with Lam\'e coefficients not depending on $\bu$).
\end{remark}

In this linear setting, the Gradient Scheme \eqref{grad-scheme} takes the form
\be\label{grad-scheme:lin}
\begin{array}{l}
\dsp \mbox{Find $\bu\in \X{\disc}$ such that,
$\forall \bv\in \X{\disc}$,}\\[0.5em]
\ba
\dsp \int_\Omega \stiff(x)\strain_\disc(\bu)(x):\strain_\disc(\bv)(x)\d x&\dsp =
\int_\Omega \bF(x)\cdot\Pi_\disc\bv(x)\d x\\
&\dsp\qquad\qquad+\int_{\Gn}\bg(x)\cdot\T_\disc(\bv)(x)\d S(x).
\ea
\end{array}\ee
The proof of the following error estimate is an adaptation of
similar estimates done in \cite{EYM11} for linear diffusion equations
with homogeneous Dirichlet boundary conditions.

\begin{theorem}[Error estimate of Gradient Scheme for linear elasticity]\label{thm:error-est}
We assu\-me that \eqref{hyp:omega}, \eqref{hyp:fg} and \eqref{hyp:stiff} hold and
we let $\buex$ be the solution
to \eqref{eq:elas-weak} (note that $\stress(x,\strain(\buex))=
\stiff\strain(\buex) \in \H_{\div}(\Omega,\Gn)^d$ since $\bF\in \L^2(\Omega)$ and $\bg\in \L^2(\Gn)$).

If $\disc$ is a Gradient Discretization in the sense of Definition \ref{def:grad-disc}
then the Gradient Scheme \eqref{grad-scheme:lin} has a unique solution
$\bu_\disc$ and it satisfies:
\begin{align}
& \Vert \nabla \buex - \nabla_\disc \bu_\disc\Vert_{\L^2(\O)^{d}}
\le \frac{K_\disc^2}{\sigma_*} W_\disc(\stiff \strain(\buex)) +
\left(\frac{K_\disc^2 \sigma^*}{\sigma_*}+1\right)S_\disc(\buex),
\label{err:grad}\\
& \Vert \buex -\Pi_\disc  \bu_\disc\Vert_{\L^2(\O)}
\le   \frac{C_\disc K_\disc^2}{\sigma_*} W_\disc(\stiff\strain(\buex))
+\left(\frac{C_\disc K_\disc^2 \sigma^*}{\sigma_*}+1\right) S_\disc(\buex),
\label{err:u}
\end{align}
where $S_\disc$, $W_\disc$, $C_\disc$ and $K_\disc$
are defined by \eqref{def:sdisc}, \eqref{def:wdisc}, \eqref{def:Cdisc}
and \eqref{def:Kdisc}.
\end{theorem}

\begin{proof}
Let us first notice that if we prove \eqref{err:grad} for any solution $\bu_\disc$
to the Gradient Scheme \eqref{grad-scheme:lin}, then the existence and uniqueness
of this solution follows. Indeed, \eqref{grad-scheme:lin}
defines a square linear system and if $\bF=0$ and $\bg=0$ (meaning that $\buex=0$) then
\eqref{err:grad} shows that the only solution to this system is $0$,
since $||\nabla_\disc\cdot||_{\L^2(\Omega)^{d}}$ is a norm on $\X{\disc}$.
Hence, this system is invertible and \eqref{grad-scheme:lin} has a solution
for any right-hand side functions $\bF$ and $\bg$ satisfying \eqref{hyp:fg}.

Let us now prove the error estimates. Since $\stiff\strain(\buex)\in \H_{\div}(\Omega,\Gn)^d$,
the definition of $W_\disc$ gives, for any $\bv\in \X{\disc}$,
\begin{equation}
\label{proof:err1}
\begin{array}{r@{}l}
\dsp ||\nabla_\disc \bv||_{\L^2(\Omega)^{d}}&W_\disc(\stiff\strain(\buex))\\[0.5em]
\ge&{}\dsp \left\vert \int_\O \big(\nabla_\disc \bv(x):\stiff(x)\strain(\buex)(x)
+\Pi_\disc\bv(x)\cdot\div(\stiff \strain(\buex)\big)(x)\d x\right.\\[0.5em]
&\dsp \left. -\int_{\Gn}\gamma_\n(\stiff\strain(\buex))(x)\cdot \T_\disc(\bv)(x)\d S(x)\right\vert\\[0.5em]
\ge&{}\dsp \left\vert \int_\O \big(\nabla_\disc \bv(x):\stiff(x)\strain(\buex)(x)
-\Pi_\disc\bv(x)\cdot  \bF(x)\big)\d x\right.\\[0.5em]
&\dsp \left.-\int_{\Gn}\bg(x)\cdot \T_\disc(\bv)(x)\d S(x)\right\vert.
\ea
\end{equation}

By symmetry of $\stiff$ we have $\stiff \strain(\buex):\nabla_\disc\bv
=\stiff\strain(\buex):\strain_\disc(\bv)$
and \eqref{proof:err1} therefore gives,
since $\bu_\disc$ is a solution to \eqref{grad-scheme:lin},
\begin{multline}\label{eq:rap}
||\nabla_\disc \bv||_{\L^2(\Omega)^{d}}
W_\disc(\stiff\strain(\buex))\\
\ge\left\vert \int_\O \stiff(x)\strain(\buex)(x):\strain_\disc(\bv)(x)
-\stiff(x)\strain_\disc(\bu_\disc):\strain_\disc(\bv)(x)\d x\right\vert.
\end{multline}
Defining, for all $\bvarphi\in \H^1_{\Gd}(\Omega)$,
\begin{equation}\label{def:PD}
P_{\disc} \bvarphi = \argmin_{\bw\in \X{\disc}}\big\{\Vert \Pi_\disc \bw -\bvarphi\Vert_{\L^2(\O)} +
\Vert \nabla_\disc \bw - \nabla\bvarphi\Vert_{\L^2(\O)^{d}}\big\}
\end{equation}
and recalling the definition \eqref{def:sdisc} of $S_\disc$, we have
\begin{equation}
||\strain(\buex)-\strain_\disc(P_\disc\buex)||_{\L^2(\Omega)^{d}}
\le ||\nabla\buex-\nabla_\disc (P_\disc\buex)||_{\L^2(\Omega)^{d}}
\le S_\disc(\buex).
\label{proof:err1.1}\end{equation}
Using the bound of $\stiff$ in \eqref{hyp:stiff} and Estimate \eqref{eq:rap}, we deduce
\begin{equation}
\begin{array}{r@{}l}
\dsp\Bigg\vert \int_\Omega \stiff(x)&\dsp\strain_\disc(P_\disc\buex-\bu_\disc)(x):\strain_\disc(\bv)(x)
\d x\Bigg\vert\\[0.5em]
\le&{}\dsp \left\vert \int_\Omega \stiff(x)(\strain_\disc(P_\disc\buex)-\strain(\buex))(x):\strain_\disc(\bv)(x)
\d x\right\vert\\[1em]
&{}\dsp+\left\vert \int_\Omega \stiff(x)(\strain(\buex)-\strain_\disc(\bu_\disc))(x):\strain_\disc(\bv)(x)
\d x\right\vert\\[1em]
\le&\dsp\; \sigma^*S_\disc(\buex)||\strain_\disc(\bv)||_{\L^2(\Omega)^{d}}
+||\nabla_\disc \bv||_{\L^2(\Omega)^{d}}W_\disc(\stiff\strain(\buex))\\[0.5em]
\le&{}\dsp\; ||\nabla_\disc \bv||_{\L^2(\Omega)^{d}}
\big(W_\disc(\stiff\strain(\buex)) +\sigma^*S_\disc(\buex)\big).
\end{array}
\label{proof:err1.2}
\end{equation}
Plugging $\bv=P_\disc\buex-\bu_\disc\in \X{\disc}$
in \eqref{proof:err1.2} and
using the coercivity of $\stiff$ gives
\begin{multline}\label{proof:err2}
\sigma_*||\strain_\disc(P_\disc\buex-\bu_\disc)||_{\L^2(\Omega)^{d}}^2\\
\le ||\nabla_\disc (P_\disc\buex-\bu_\disc)||_{\L^2(\Omega)^{d}}
\big(W_\disc(\stiff\strain(\buex)) +\sigma^*S_\disc(\buex)\big).
\end{multline}
By definition \eqref{def:Kdisc} of $K_\disc$, we have
\[
||\nabla_\disc(P_\disc\buex-\bu_\disc)||_{\L^2(\Omega)^{d}}\le
K_\disc ||\strain_\disc(P_\disc\buex-\bu_\disc)||_{\L^2(\Omega)^{d}}
\]
and \eqref{proof:err2} thus leads to
\begin{equation}\label{proof:err3}
||\nabla_\disc (P_\disc\buex)-\nabla_\disc\bu_\disc||_{\L^2(\Omega)^{d}}
\le \frac{K_\disc^2}{\sigma_*}
\big(W_\disc(\stiff\strain(\buex)) +\sigma^*S_\disc(\buex)\big)
\end{equation}
and the proof of \eqref{err:grad} is concluded thanks to \eqref{proof:err1.1}.
The Poincar\'e inequality \eqref{ineq:poincare} and \eqref{proof:err3}
also give
\[
||\Pi_\disc (P_\disc\buex)-\Pi_\disc \bu_\disc||_{\L^2(\Omega)}
\le \frac{C_\disc K_\disc^2}{\sigma_*}
\big(W_\disc(\stiff\strain(\buex)) +\sigma^*S_\disc(\buex)\big),
\]
and the estimate $||\Pi_\disc (P_\disc\buex)-\buex||_{\L^2(\Omega)}
\le S_\disc(\buex)$ concludes the proof of \eqref{err:u}. \end{proof}

The following corollary is a straightforward consequence of Theorem \ref{thm:error-est}.

\begin{corollary}[Convergence of Gradient Schemes for linear
elasticity]
We assume that \eqref{hyp:omega}, \eqref{hyp:fg} and \eqref{hyp:stiff} hold.
We denote by $\buex$ the solution
to \eqref{eq:elas-weak}.

If $(\disc_m)_{m\in\N}$ is a sequence of Gradient Discretizations
in the sense of Definition \ref{def:grad-disc}, which is
consistent (Definition \ref{def:cons}), limit-conforming (Definition \ref{def:lim-conf})
and coercive (Definition \ref{def:coer}), and if $\bu_m\in \X{\disc_m}$
is the solution to the Gradient Scheme \eqref{grad-scheme:lin} with $\disc=\disc_m$,
then, as $m\to\infty$, $\Pi_{\disc_m} \bu_m\to \buex$ strongly in $\L^2(\Omega)$
and $\nabla_{\disc_m}\bu_m \to \nabla\buex$ strongly in $\L^2(\Omega)^{d}$.
\end{corollary}

\begin{remark}
This result is valid under no additional regularity assumption on the data
or $\buex$. It holds in particular if $\partial\Omega$ has singularities
or if $\stiff$ is discontinuous with respect
to $x$, which corresponds to a body made of several different materials
with interfaces (see e.g. \cite{BIS09}).

However, for most Gradient Schemes (and under reasonable assumptions on the mesh/discretisation),
there exists $C>0$ not depending on $\disc$ such that
\begin{eqnarray*}
\forall \bvarphi\in \H^2(\Omega)\cap \H^1_0(\Omega)\,,\;&&
S_\disc(\bvarphi)\le Ch_\disc||\bvarphi||_{\H^2(\Omega)}\,,\\
\forall \bL\in \H^1(\Omega)^{d}\,,\;&&
 W_\disc(\bL)\le Ch_\disc||\bL||_{\H^1(\O)^{d}},
\end{eqnarray*}
where $h_\disc$ measures the scheme's precision (e.g. some mesh size).
For such Gradient Schemes and when $\bu\in \H^2(\O)$ and $\stiff$ is Lipschitz continuous,
Theorem \ref{thm:error-est} gives $\mathcal O(h_\disc)$ error
estimate for the approximation of $\buex$ and its gradient.
We note that the solution is $H^2$-regular when we have a pure Dirichlet problem on a 
convex polygonal or polyhedral domain \cite{BS92,KMR01}. 
\end{remark}

\subsection{Non-linear case}

In the non-linear case, error estimates cannot be provided in general
but convergence of the Gradient Scheme \eqref{grad-scheme} can
still be proved without additional regularity assumptions on the data.

\begin{theorem}[Convergence of Gradient Schemes for non-linear
elasticity]\label{th:convnl}
Assu\-me that \eqref{hyp:omega}, \eqref{hyp:fg} and \eqref{hyp:stress} hold
and let $(\disc_m)_{m\in\N}$ be a sequence of Gradient Discretizations
in the sense of Definition \ref{def:grad-disc}, which is
consistent (Definition \ref{def:cons}), limit-conforming (Definition \ref{def:lim-conf})
and coercive (Definition \ref{def:coer}).

Then, for any $m\in\N$ there exists at least one solution $\bu_m\in
\X{\disc_m}$ to the Gradient Scheme \eqref{grad-scheme} with $\disc=\disc_m$ and,
up to a subsequence, as $m\to\infty$, $\Pi_{\disc_m}\bu_m$
converges weakly in $\L^2(\Omega)$ to some $\buex$ solution of \eqref{eq:elas-weak}
and $\nabla_{\disc_m}\bu_m$ converges weakly in $\L^2(\Omega)^{d}$
to $\nabla\buex$.

Moreover, if we assume that $\stress$ is strictly monotone in the following sense: 
\begin{equation}\label{hyp:strictmonotone}
\mbox{For a.e. $x\in\Omega$, for all $\bL\not=\bG$ in $\sym$}\,,\;
(\stress(x,\bL)-\stress(x,\bG)):(\bL-\bG)>0
\end{equation}
then, along the same subsequence, $\Pi_{\disc_m}\bu_m\to \buex$
strongly in $\L^2(\Omega)$ and
$\nabla_{\disc_m}\bu_m\to \nabla\buex$ strongly in $\L^2(\Omega)^{d}$.
\end{theorem}

\begin{remark}
If the sequence of Gradient Discretisations $(\disc_m)_{m\in\N}$
is \emph{compact} as defined in \cite{DRO12}, then the convergence of $\Pi_{\disc_m}\bu_m$
is strong even if the strict monotonicity \eqref{hyp:strictmonotone} is not satisfied.
\end{remark}

\begin{remark} Should the solution to \eqref{eq:elas-weak} be unique,
classical arguments also show that the convergences of $(\bu_m)_{m\in\N}$ in the senses
described in Theorem \ref{th:convnl} hold for the whole
sequence, not only for a subsequence.
\end{remark}

\begin{remark} We do not need to assume the existence of a solution
to the non-linear elasticity model \eqref{eq:elas-weak}.
The technique of convergence analysis we use establishes in fact this existence.
\end{remark}

\begin{remark} The strict monotonicity assumption \eqref{hyp:strictmonotone} is
satisfied by the Hencky-von Mises model (see \cite[Lemma 4.1]{BAR02}),
and by the damage model when the function $f$ defined in Remark \ref{rem:dam}
is such that $s\in[0,\infty)\to sf(s)$ is \emph{strictly} increasing.
\end{remark}

\begin{proof}
The proof follows the techniques used in \cite{DRO12} for the non-linear
elliptic problem with homogeneous Dirichlet boundary conditions. We adapt
those techniques to deal with the non-linear elasticity models with
mixed non-homogeneous boundary conditions.
In the following steps, we sometimes drop the index $m$ in $\disc_m$ to
simplify the notations.

\textbf{Step 1}: A priori estimates and existence of a solution to the scheme.

Let us take a scalar product $\langle\cdot,\cdot\rangle$ on $\X{\disc}$,
with associated norm $N(\cdot)$, and let us define
$T:\X{\disc}\to \X{\disc}$ and $L\in \X{\disc}$ by: for all $\bw,\bv\in \X{\disc}$,
\[
\langle T(\bw),\bv\rangle=\int_\Omega \stress(x,\strain_\disc(\bw)(x)):
\strain_\disc(\bv)(x)\d x
\]
and
\[
\langle L,\bv\rangle= \int_\Omega \bF(x)\cdot\Pi_\disc\bv(x)\d x+\int_{\Gn}\bg(x)\cdot\T_\disc(\bv)(x)\d S(x).
\]
Then Assumption \eqref{hyp:stress} ensures that $T$ is continuous
and that 
\begin{equation}\label{prnl:est1}
\langle T(\bw),\bw\rangle \ge
\sigma_*||\strain_\disc(\bw)||_{\L^2(\Omega)^{d}}^2
\ge\sigma_* K_{\disc}^{-2} ||\bw||_\disc^2.
\end{equation}
Since all norms are equivalent on $\X{\disc}$, we also have $||\bw||_\disc\ge
m_\disc N(\bw)$ for some $m_\disc>0$ and this shows that
$\lim_{N(\bw)\to\infty}\frac{\langle T(\bw),\bw\rangle}{N(\bw)}=+\infty$.
By \cite[Theorem 3.3 (p.19)]{DEI85}, we see that $T$ is onto and therefore
that there exists $\bu\in \X{\disc}$ such that $T(\bu)=L$, which precisely
states that $\bu$ is a solution to \eqref{grad-scheme}.

{}From \eqref{prnl:est1} and the definition \eqref{def:Cdisc} of $C_\disc$,
we also deduce that $\bu$ satisfies
\begin{multline*}
||\bu||_\disc^2 \le \frac{K_\disc^2}{\sigma_*}\langle T(\bu),\bu\rangle
=\frac{K_\disc^2}{\sigma_*}\langle L,\bu\rangle\\
\le \frac{K_\disc^2}{\sigma_*}||\bF||_{\L^2(\Omega)}||\Pi_\disc\bu||_{\L^2(\Omega)}
+\frac{K_\disc^2}{\sigma_*}||\bg||_{\L^2(\Gn)}||\T_\disc\bu||_{\L^2(\Gn)}\\
\le \left(\frac{C_\disc K_\disc^2}{\sigma_*}||\bF||_{\L^2(\Omega)}
+\frac{C_\disc K_\disc^2}{\sigma_*}||\bg||_{\L^2(\Gn)}\right)||\bu||_\disc,
\end{multline*}
that is to say
\be\label{prnl:est2}
||\bu||_\disc
\le \frac{C_\disc K_\disc^2}{\sigma_*}||\bF||_{\L^2(\Omega)}
+\frac{C_\disc K_\disc^2}{\sigma_*}||\bg||_{\L^2(\Gn)}.
\ee

\medskip

\textbf{Step 2}: Weak convergences.

By Estimate \eqref{prnl:est2}, $(||\bu_m||_{\disc_m})_{m\in\N}$
is bounded and Lemma \ref{lem:weakconv} below therefore shows that
there exists $\buex\in \H^1_{\Gd}(\Omega)$ such that, up to a subsequence,
\be\label{prnl:conv}
\ba
\dsp \Pi_{\disc_m}\bu_m\to \buex\mbox{ weakly in $\L^2(\Omega)$}\,,\\
\dsp \nabla_{\disc_m}\bu_m\to \nabla\buex\mbox{ weakly in $\L^2(\Omega)^{d}$ and}\\
\dsp \T_{\disc_m}\bu_m\to \gamma(\buex)\mbox{ weakly in $\L^2(\Gn)$.}\\
\ea
\ee

Let us now prove that $\buex$ is a solution to \eqref{eq:elas-weak}.
Assumptions \eqref{hyp:stress} and the bound on $\nabla_{\disc_m}\bu_m$
shows that $(\stress(\cdot,\strain_{\disc_m}(\bu_m))_{m\in\N}$ is symmetric-valued and bounded
in $\L^2(\Omega)^{d}$. There exists therefore a symmetric-valued $\bL\in \L^2(\Omega)^{d}$
such that, up to a subsequence,
\be\label{prnl:conv1}
\stress(\cdot,\strain_{\disc_m}(\bu_m))\to \bL\mbox{ weakly in $\L^2(\Omega)^{d}$}.
\ee

Let $\bvarphi\in \H^1_{\Gd}(\Omega)$. Then $P_{\disc_m}\bvarphi$
defined by \eqref{def:PD} belongs to $\X{\disc_m}$ and, by consistency of $(\disc_m)_{m\in\N}$,
$\Pi_{\disc_m}(P_{\disc_m}\bvarphi)\to \bvarphi$ strongly in
$\L^2(\Omega)$ and $\nabla_{\disc_m}(P_{\disc_m}\bvarphi)
\to \nabla\bvarphi$ strongly in $\L^2(\Omega)^{d}$. By Lemma
\ref{lem:weakconv}, we also deduce that $\T_{\disc_m}(P_{\disc_m}\bvarphi)
\to \gamma(\bvarphi)$ weakly in $\L^2(\Gn)$.
The convergence \eqref{prnl:conv1} then allows to pass to the limit
in \eqref{grad-scheme} with $\bv=P_{\disc_m}\bvarphi$ as a test function and
we obtain
\be\label{prnl:conv2}
\int_\Omega \bL(x):\strain(\bvarphi)(x)\d x=\int_\Omega \bF(x)\cdot\bvarphi(x)\d x
+\int_{\Gn}\bg(x)\cdot\gamma(\bvarphi)(x)\d S(x).
\ee

We now use the monotonicity assumption on $\stress$ and Minty's trick \cite{MIN63,LER65}
to prove that $\bL=\stress(x,\strain(\buex))$.
We first notice that, plugging $\bv=\bu_m$ in \eqref{grad-scheme}
and using \eqref{prnl:conv} and \eqref{prnl:conv2},
\begin{multline}
\label{prnl:conv3}
\int_\Omega \stress(x,\strain_{\disc_m}(\bu_m)(x)):\strain_{\disc_m}(\bu_m)(x)\d x\\
=\int_\Omega \bF(x)\cdot\Pi_{\disc_m}\bu_m(x)\d x+\int_{\Gn}\bg(x)\cdot \T_{\disc_m}\bu_m(x)\d S(x)\\
\longrightarrow\int_\Omega \bF(x)\cdot \buex(x)\d x+\int_{\Gn}\bg(x)\cdot \gamma(\buex)(x)\d S(x)
=\int_\Omega \bL(x):\strain(\buex)(x)\d x.
\end{multline}
Let us now take any symmetric-valued $\bG\in \L^2(\Omega)^{d}$. The monotonicity of $\stress$ shows that
\[
A_m:=\int_\Omega \big[\stress(x,\strain_{\disc_m}(\bu_m)(x))-
\stress(x,\bG(x))\big]:\big[\strain_{\disc_m}(\bu_m)(x)-\bG(x)\big]\d x\ge 0.
\]
After developing $A_m$, we can use \eqref{prnl:conv}, \eqref{prnl:conv1} and \eqref{prnl:conv3}
to pass to the limit and we find
\begin{equation}\label{prnl:conv24}
\lim_{m\to\infty}A_m=
\int_\Omega \big[\bL(x)-\stress(x,\bG(x))\big]
:\big[\strain(\buex)(x)-\bG(x)\big]\d x\ge 0.
\end{equation}
The Minty trick then concludes the proof. Applying this inequality to
$\bG=\strain(\buex)+\alpha\bD$ for some symmetric-valued $\bD\in \L^2(\Omega)^{d}$,
dividing by $\alpha$ and letting $\alpha\to 0^{\pm}$ (thanks to Assumption \eqref{hyp:stress}), we 
obtain 
\[
\int_\Omega \big[\bL(x)-\stress(x,\strain(\buex)(x))\big]
:\bD(x)\d x=0,
\]
which proves, with $\bD=\bL-\stress(\cdot,\strain(\buex))$, that
\be\label{prnl:ident}
\bL=\stress(\cdot,\strain(\buex)).
\ee
Together with \eqref{prnl:conv2} this shows that $\buex$ satisfies
\eqref{eq:elas-weak}.

\medskip

\textbf{Step 3}: Strong convergences under strict monotonicity.

We now assume that \eqref{hyp:strictmonotone} holds and we
first prove the strong convergence of the strain tensors.
We define
\[
f_m=\big[\stress(\cdot,\strain_{\disc_m}(\bu_m))-
\stress(\cdot,\strain(\buex))\big]\\
:\big[\strain_{\disc_m}(\bu_m)-\strain(\buex)\big].
\]
The function $f_m$ is non-negative and, by \eqref{prnl:conv24}
with $\bG=\strain(\buex)$ and the identity \eqref{prnl:ident}, we see that
$\lim_{m\to\infty}\int_\Omega f_m(x)\d x=0$. $(f_m)_{m\in\N}$ thus converges
to $0$ in $L^1(\Omega)$, and therefore also a.e. on $\Omega$ up to a subsequence.

Let us take $x\in\Omega$ such that the above mentioned convergence hold at $x$.
{}From the coercivity and growth of $\stress$, developing the
products in $f_m(x)$ gives
\begin{multline*}
f_m(x)\ge \sigma_*|\strain_{\disc_m}(\bu_m)(x)|^2
-2\sigma^*|\strain_{\disc_m}(\bu_m)(x)|\, |\strain(\buex)(x)|\\
-|\sigma(x,\strain(\buex)(x)|\,|\strain(\buex)(x)|.
\end{multline*}
Since the right-hand side is quadratic in $|\strain_{\disc_m}(\bu_m)(x)|$
and $(f_m(x))_{m\in\N}$ is bounded,
we deduce that the sequence $(\strain_{\disc_m}(\bu_m)(x))_{m\in\N}$
is bounded. If $\mathbf{L}_x$ is one of its adherence values then, by passing to the limit
in the definition of $f_m(x)$, we see that
\[
0=\big[\stress(x,\mathbf{L}_x)-
\stress(x,\strain(\buex)(x))\big]
:\big[\mathbf{L}_x-\strain(\buex)(x)\big].
\]
By \eqref{hyp:strictmonotone}, this forces $\mathbf{L}_x=\strain(\buex)(x)$.
The bounded sequence $(\strain_{\disc_m}(\bu_m)(x))_{m\in\N}$
only has $\strain(\buex)(x)$ as adherence value and therefore
converges in whole to this value. We have therefore established that
$\strain_{\disc_m}(\bu_m)\to \strain(\bu)$ a.e. on $\Omega$.

Using then \eqref{prnl:conv3} and \eqref{prnl:ident} and
defining
\[
F_m=\stress(\cdot,\strain_{\disc_m}(\bu_m)):\strain_{\disc_m}(\bu_m)\ge 0,
\]
we see that
\[
\lim_{m\to\infty}\int_\Omega F_m(x)\d x=
\int_\Omega \stress(x,\strain(\buex)(x)):\strain(\buex)(x)\d x.
\]
But since $F_m\to \stress(\cdot,\strain(\buex)):\strain(\buex)$ a.e.
on $\Omega$ and is non-negative, we can apply Lemma \ref{lem:convpositive}
below to deduce that $(F_m)_{m\in\N}$ converges in $L^1(\Omega)$.
This sequence is therefore equi-integrable in $L^1(\Omega)$ and,
by the coercivity property of $\stress$, this proves that
$(\strain_{\disc_m}(\bu_m))_{m\in\N}$ is
equi-integrable in $\L^2(\Omega)^{d}$. As this sequence converges
a.e. on $\Omega$ to $\strain(\buex)$, Vitali's theorem shows that
\be\label{prnl:conv5}
\strain_{\disc_m}(\bu_m)\to \strain(\buex)\mbox{ strongly in $\L^2(\Omega)^{d}$}.
\ee
We then consider $P_{\disc_m}\buex\in \X{\disc_m}$ and write,
by definition \eqref{def:Kdisc} of $K_\disc$,
\[
||\nabla_{\disc_m}\bu_m-\nabla_{\disc_m}(P_{\disc_m}\buex)||_{\L^2(\Omega)^{d}}
\le K_{\disc_m}||\strain_{\disc_m}(\bu_m)-\strain_{\disc_m}(P_{\disc_m}\buex)||_{\L^2(\Omega)^{d}}.
\]
Since $\nabla_{\disc_m}(P_{\disc_m}\buex)$ and $\strain_{\disc_m}(P_{\disc_m}\buex)$
strongly converge in $\L^2(\Omega)^d$ to $\nabla\buex$ and $\strain(\buex)$,
we can pass to the limit in this estimate by using the coercivity of
$(\disc_m)_{m\in\N}$ and \eqref{prnl:conv5} and we deduce that
$\nabla_{\disc_m}\bu_m\to \nabla\buex$ strongly in $\L^2(\Omega)^{d}$.
The definition \eqref{def:Cdisc} of $C_{\disc_m}$ then gives
\[
||\Pi_{\disc_m}\bu_m-\Pi_{\disc_m}(P_{\disc_m}\buex)||_{\L^2(\Omega)}
\le C_{\disc_m}
||\nabla_{\disc_m}\bu_m-\nabla_{\disc_m}(P_{\disc_m}\buex)||_{\L^2(\Omega)^{d}}
\]
and, since $\Pi_{\disc_m}(P_{\disc_m}\buex)\to \buex$ strongly in $\L^2(\Omega)$,
passing to the limit in this estimate proves the strong
convergence in $\L^2(\Omega)^d$ of $\Pi_{\disc_m}\bu_m$ to $\buex$. \end{proof}

\begin{remark} We saw in the proof that $\T_{\disc_m}\bu_m\to \gamma(\buex)$
weakly in $\L^2(\Gn)$. If the interpolation $P_\disc$ defined by
\eqref{def:PD} satisfies, for any $\bvarphi\in \H^1_{\Gd}(\Omega)$,
$\T_{\disc_m}(P_{\disc_m}\bvarphi)\to \gamma(\bvarphi)$ strongly in $\L^2(\Gn)$ as $m\to\infty$,
the same reasoning as the one used at the end of the proof shows
that, in case of strict monotonicity of $\stress$,
$\T_{\disc_m}\bu_m\to \gamma(\buex)$ strongly in $\L^2(\Gn)$.
\end{remark}

\begin{lemma}\label{lem:weakconv}
Let $(\disc_m)_{m\in\N}$ be a sequence of Gradient Discretizations
in the sense of Definition \ref{def:grad-disc}, which is
limit-conforming (Definition \ref{def:lim-conf}) and coercive (Definition \ref{def:coer}).
For any $m\in\N$ we take $\bv_m\in \X{\disc_m}$.

If $(||\bv_m||_{\disc_m})_{m\in\N}$ is
bounded then there exists $\bv\in \H^1_{\Gd}(\Omega)$ such that, up
to a subsequence, $\Pi_{\disc_m}\bv_m\to \bv$ weakly in $\L^2(\Omega)$,
$\nabla_{\disc_m}\bv_m\to \nabla\bv$ weakly in $\L^2(\Omega)^d$
and $\T_{\disc_m}\bv_m\to \gamma(\bv)$ weakly in $\L^2(\Gn)$.
\end{lemma}

\begin{proof}
The coercivity of $(\disc_m)_{m\in\N}$ and the bound on $||\bv_m||_{\disc_m}$
show that the sequences $||\Pi_{\disc_m}\bv_m||_{\L^2(\Omega)}$, $||\nabla_{\disc_m}\bv_m||_{\L^2(\Omega)^d}$
and $||\T_{\disc_m}\bv_m||_{\L^2(\Gn)}$ remain bounded. There exists therefore
$\bv\in \L^2(\Omega)$, $\bG\in \L^2(\Omega)^d$ and $\bw\in \L^2(\Gn)$ such that,
up to a subsequence,
\be\label{lem:weakconv1}
\ba
\dsp \Pi_{\disc_m}\bv_m\to \bv\mbox{ weakly in $\L^2(\Omega)$}\,,\quad
\nabla_{\disc_m}\bv_m\to \bG\mbox{ weakly in $\L^2(\Omega)^{d}$ and}\\
\dsp \T_{\disc_m}\bv_m\to \bw\mbox{ weakly in $\L^2(\Gn)$.}\\
\ea
\ee
These convergences and the limit-conformity
of $(\disc_m)_{m\in\N}$ show that, for any $\bL\in \H_{\div}(\Omega,\Gn)^d$,
\begin{multline*}
\left|\int_\Omega \bG(x):\bL(x)+\bv(x)\cdot\div(\bL)(x)\d x - \int_{\Gn}\gamma_\n(\bL)(x)\cdot\bw(x)\d S(x)\right|\\
=\lim_{m\to\infty}
\left|\int_\Omega \nabla_{\disc_m}\bv_m(x):\bL(x)+\Pi_{\disc_m}\bv_m(x)\cdot\div(\bL)(x)\d x\right.\\
\left. - \int_{\Gn}\gamma_\n(\bL)(x)\cdot \T_{\disc_m}(\bv_m)(x)\d S(x)\right|\\
\le \lim_{m\to\infty}\big[||\bv_m||_{\disc_m}W_{\disc_m}(\bL)\big]=0.
\end{multline*}
Hence, for any $\bL\in \H_{\div}(\Omega,\Gn)^d$,
\begin{equation}\label{lem:weakconv2}
\int_\Omega \bG(x):\bL(x)+\bv(x)\cdot\div(\bL)(x)\d x- \int_{\Gn}\gamma_\n(\bL)(x)\cdot\bw(x)\d S(x)=0.
\end{equation}
Applied with $\bL\in C^\infty_c(\Omega)^{d\times d}$, this relation shows that
\be\label{lem:weakconv3}
\nabla \bv=\bG\mbox{ in the sense of distributions on $\Omega$},
\ee
and thus that $\bv\in \H^1(\Omega)$. By using \eqref{lem:weakconv2} with $\bL\in \H^1(\Omega)^{d}\subset
\H_{\div}(\Omega,\Gn)^d$ and by integrating by parts, we obtain
\[
\int_{\partial\Omega} \gamma_\n(\bL)(x)\cdot \gamma(\bv)(x)\d S(x)
- \int_{\Gn}\gamma_\n(\bL)(x)\cdot\bw(x)\d S(x)=0.
\]
As the set $\{\gamma_\n(\bL)\,:\,\bL\in \H^1(\Omega)^{d}\}$ is dense
in $\L^2(\partial\Omega)$, we deduce from this that $\gamma(\bv)=0$ on $\Gd$ and that
\be\label{lem:weakconv4}
\gamma(\bv)=\bw\mbox{ on $\Gn$.}
\ee
Thus, $\bv\in \H^1_{\Gd}$ and \eqref{lem:weakconv1}, \eqref{lem:weakconv3} and \eqref{lem:weakconv4}
conclude the proof. \end{proof}

The proof of the following lemma is classical \cite{DRO06,EYM09}.

\begin{lemma}\label{lem:convpositive} Let $(F_m)_{m\in\N}$ be a sequence
of non-negative measurable functions on $\Omega$ which converges a.e. on $\Omega$
to $F$ and such that $\int_\Omega F_m(x)\d x\to \int_\Omega F(x)\d x$.
Then $F_m\to F$ in $L^1(\Omega)$.
\end{lemma}

\section{Examples of Gradient Schemes}\label{sec:ex}

In all the following examples, we assume that 
$\Gd$ has non-zero measure and is such that a K\"orn's inequality holds on $\H^1_{\Gd}(\O)$
\cite{BS94,Cia88}. This is actually a necessary condition for \emph{coercive}
and \emph{consistent} sequences of Gradient Discretisations to exist.

\subsection{Standard displacement-based formulation} 

All (conforming) Galerkin methods are Gradient Schemes.
If $(\bV_n)_{n\in\N}$ is a sequence of finite dimensional subspaces
of $\H^1_{\Gd}(\O)$ such that $\cup_{n\ge 1}\bV_n$ is dense in $\H^1_{\Gd}(\O)$,
then by letting $X_{\disc_n}=V_n$,  $\Pi_{\disc_n}={\rm Id}$,
$\T_{\disc_n}=\gamma$ and $\nabla_{\disc_n}=\nabla$, we obtain
a sequence of Gradient Discretisations whose corresponding Gradient Schemes
are Galerkin approximations of \eqref{eq:elas-strong}. This
sequence of Gradient Discretisations is obviously \emph{consistent}
(this is $\overline{\cup_{n\in\N}\bV_n}=\H^1_{\Gd}(\O)$),
\emph{limit-conforming} (as it is a conforming approximation,
$W_{\disc_n}=0$ for any $n$) and \emph{coercive} (since Poincar\'e's and
K\"orn's inequalities hold in $\H^1_{\Gd}(\O)$).

This is in particular the case for conforming Finite Element approximations
based on spaces $\bV_h$ built on quasi-uniform partitions $\CT_h$ of $\O$
(made of quadrilaterals, hexahedra or simplices \cite{Bra01,QV94}).

But non-conforming methods are also included in the framework of Gradient Schemes.
For example, the Crouzeix-Raviart scheme falls in this framework, with the discrete gradient
defined as the classical ``broken gradient''. Consistency, limit-conformity
and the Poincar\'e's inequality for this scheme are established in \cite{DRO13},
and it is known that if $\Gd=\partial\O$ then a uniform K\"orn's inequality
holds. This inequality fails for general $\Gd$ \cite{FAL90} but
it is satisfied for higher order non-conforming methods (whose continuity conditions
through the edges involve both the zero-th and
first order moments) \cite{KNO00}. The consistency, limit-conformity and
Poincar\'e's inequality for such methods can be easily established as for
Crouzeix-Raviart's method.

\subsection{Stabilised nodal strain formulation} 

We consider a nodal strain formulation as presented in \cite{FB81,PS06,Lam09p}
and built on a conforming Finite Element space $\bV_h$.
Associated with the primal mesh $\CT_h$ we let 
$\CT^*_h$ be the dual mesh consisting of dual volumes, 
where a dual volume is associated with a 
vertex of $\CT_h^*$ and is constructed as follows.
Let $ \{T^{\x_i}_j\}_{j=1}^{M_i}\subset \CT_h$ be the set of all elements touching 
the vertex $\x_i$, and $ \{E^{\x_i}_j\}_{j=1}^{N_i}$ the set of 
edges or faces touching $\x_i$. Then the dual volume 
associated with the vertex $\x_i$ is the polygonal 
or polyhedral region joining all the bary-centres of 
$ \{T^{\x_i}_j\}_{j=1}^{M_i}$ 
and $ \{E^{\x_i}_j\}_{j=1}^{N_i}$ . Let $S^*_h$ be the 
space of vector-valued piecewise constant functions with respect to the dual mesh $\CT_h^*$.

Defining the linear form 
\[
\ell(\bv_h) = \int_\Omega \bF(x)\cdot\bv_h(x)\d x +\int_{\Gn}\bg(x)\cdot\gamma(\bv_h)(x)\d S(x),
\]
the stabilised nodal strain formulation, for a constant stiffness
tensor $\stiff$, is to find $\bu_h\in \bV_h$ such that, for any $\bv_h\in \bV_h$,
\[
\int_{\Omega} \Pi^*_h \strain(\bu_h)(x):\stiff \strain(\bv_h)(x)\d x+
\int_{\Omega} \stab(\strain(\bu_h)-\Pi^*_h\strain(\bu_h))(x):
\strain(\bv_h)(x)\d x =\ell(\bv_h)
\]
where $\Pi^*_h$ is the orthogonal projection onto $S^*_h$
and $\stab$ is a constant stabilisation (symmetric positive definite) tensor. By the properties of the
orthogonal projection and since $\stiff$ and $\stab$ are constant, this can be recast as
\be\label{stab-nodal}
\begin{array}{l}
\dsp \mbox{Find $\bu_h\in \bV_h$ such that,
$\forall \bv_h\in \bV_h$,}\\[0.5em]
\ba \dsp \int_{\Omega} \stiff \Pi^*_h \strain(\bu_h)(x):\Pi^*_h\strain(\bv_h)(x)\d x\\
\dsp\qquad\qquad+
 \int_{\Omega} \stab(\strain(\bu_h)-\Pi^*_h\strain(\bu_h))(x):
(\strain(\bv_h)-\Pi^*_h\strain(\bv_h))(x)\d x  =\ell(\bv_h).
\ea
\ea\ee
We will take this formulation as definition of the stabilised nodal strain
formulation in the case where $\stiff$ and $\stab$ are not constant (in which
case we assume that $\stab$ satisfies Assumption \eqref{hyp:stiff}).

Let us now construct a Gradient Discretisation $\disc=(\X{\disc},\Pi_\disc,
\T_\disc,\nabla_\disc)$ such that this formulation is identical to
the corresponding Gradient Scheme \eqref{grad-scheme:lin}.
We start by defining $\XD{\disc}$ and the operators $\Pi_\disc:\XD{\disc}\to \L^2(\O)$
and $\T_\disc:\XD{\disc}\to \L^2(\Gn)$ by
\begin{equation}\label{def:gsnodal}
\X{\disc}=\bV_h\,,\;\Pi_\disc \bv_h=\bv_h\mbox{ and }
\T_\disc \bv_h=\gamma(\bv_h)_{|\Gn}\mbox{ for all $\bv_h\in \XD{\disc}$}.
\end{equation}
With these choices, $\ell(\bv_h)$ is the right-hand side of \eqref{grad-scheme:lin}
and we therefore just need to find a discrete gradient $\nabla_\disc$
such that the left-hand side of \eqref{grad-scheme:lin} is equal
to the left-hand side of \eqref{stab-nodal}. 

We first notice that, by \eqref{hyp:stiff} on $\stiff$ and $\stab$, for a.e. $x$
the linear mappings $\stiff(x),\stab(x):\R^{d\times d}\to\R^{d\times d}$ are symmetric
positive definite  with respect to  the inner product ``$:$'' and thus
$\stiff(x)^{-1/2}$ and $\stab(x)^{1/2}$ make sense.
We can therefore define $\nabla_\disc:\X{\disc}\to \L^2(\O)^d$ by
\begin{equation}\label{def:gradnodal}
\nabla_\disc\bv_h = \Pi^*_h \nabla\bv_h + \stiff^{-1/2}\stab^{1/2}(\nabla\bv_h-\Pi^*_h\nabla\bv_h).
\end{equation}
By assumptions on $\stiff$ and $\stab$ and Lemma \ref{lem:root}, this gives
\[
\strain_\disc(\bv_h)=\Pi^*_h \strain(\bv_h) + \stiff^{-1/2}\stab^{1/2}(\strain(\bv_h)-\Pi^*_h\strain(\bv_h)).
\]
Assuming that $\stiff$ and $\stab$ are piecewise constant on $\CT_h^*$, we can then compute
\begin{eqnarray}
\lefteqn{\int_\Omega\stiff(x) \strain_\disc(\bu_h)(x):\strain_\disc(\bv_h)(x)\d x}&&\nonumber\\
&=&\int_\Omega \stiff(x)\Pi^*_h \strain(\bu_h)(x):\Pi^*_h \strain(\bv_h)(x)\d x\nonumber\\
&&+\int_\Omega \stiff(x)\Pi^*_h \strain(\bu_h)(x):\stiff^{-1/2}(x)\stab^{1/2}(x)
(\strain(\bv_h)(x)-\Pi^*_h \strain(\bv_h)(x)) \d x \label{T2}\\
&&+\int_\Omega \stiff(x)\stiff^{-1/2}(x)\stab^{1/2}(x)
(\strain(\bu_h)(x)-\Pi^*_h \strain(\bu_h)(x)):\Pi^*_h\strain(\bv_h)(x)\d x\label{T3}\\
&&+\int_\Omega \stiff(x)\stiff^{-1/2}(x)\stab^{1/2}(x)
(\strain(\bu_h)(x)-\Pi^*_h \strain(\bu_h)(x))\nonumber\\
&&\qquad\qquad:\stiff^{-1/2}(x)\stab^{1/2}(x)
(\strain(\bv_h)(x)-\Pi^*_h \strain(\bv_h)(x)) \d x.
\nonumber
\end{eqnarray}
But, since $\stiff$, $\stab$ and $\Pi^*_h\strain(\bu_h)$ are constant on each cell in $\CT^*_h$
and since 
\[ \Pi_h^*\strain(\bv_h)=\frac{1}{{\rm meas}(K)}\int_K \strain(\bv_h)(x)\d x\]
on $K\in \CT_h^*$, we have
\[
\eqref{T2}=\sum_{K\in\CT_h^*}\stiff_{|K} 
\Pi^*_h\strain(\bu_h)_{|K}:\stiff^{-1/2}_{|K}\stab^{1/2}_{|K} 
\int_K \left(\strain(\bv_h)(x)-\Pi^*_h\strain(\bv_h)(x)\right)\d x=0.
\]
Similarly, \eqref{T3} vanishes and, by using the symmetry of $\stiff$ and $\stab$, we end up with
\begin{eqnarray*}
\lefteqn{\int_\Omega\stiff(x) \strain_\disc(\bu_h)(x):\strain_\disc(\bv_h)(x)\d x}&&\\
&=&\int_\Omega \stiff(x)\Pi^*_h \strain(\bu_h)(x):\Pi^*_h \strain(\bv_h)(x)\\
&&+ \int_\Omega \stab(x)(\strain(\bu_h)(x)-\Pi^*_h \strain(\bu_h)(x))
:(\strain(\bv_h)(x)-\Pi^*_h \strain(\bv_h)(x)) \d x,
\end{eqnarray*}
which precisely states that the left-hand sides of \eqref{grad-scheme:lin}
and \eqref{stab-nodal} coincide. 
Thus, under the assumption that $\stiff$ and $\stab$ are piecewise constant on $\CT_h^*$,
the stabilised nodal strain formulation \eqref{stab-nodal} is the
Gradient Scheme, for the linear elasticity equation, corresponding to
the Gradient Discretisation defined by \eqref{def:gsnodal}--\eqref{def:gradnodal}.

\begin{remark}\label{ns:stiff} If $\stiff$ or $\stab$ are not piecewise constant on $\CT_h^*$,
then by replacing them with $\Pi_h^*\stiff$ and $\Pi_h^*\stab$
in the stabilised nodal strain formulation \eqref{stab-nodal} and the
definition \eqref{def:gradnodal} of the discrete gradient,
the stabilised nodal strain formulation is the Gradient Scheme 
\eqref{grad-scheme:lin} in which $\stiff$ is replaced with $\Pi_h^*\stiff$.
\end{remark}

\subsubsection{Consistency, limit-conformity and coercivity}\label{sec:propstabnodal}

Let us consider $(\bV_{h_n})_{n\in\N}$ a sequence of conforming Finite Element
spaces on meshes $(\CT_{h_n})_{n\in\N}$ with $h_n\to 0$. We prove here
that if $\disc_n$ is the Gradient Discretisation given by
\eqref{def:gsnodal}--\eqref{def:gradnodal} for $\bV_{h_n}$ then, under
the classical quasi-uniform assumptions on $(\CT_{h_n})_{n\in\N}$,
the sequence $(\disc_n)_{n\in\N}$ is consistent, limit-conforming and coercive. 
The key point is to notice that the definition \eqref{def:gradnodal} of the
discrete gradient can be recast as
\begin{equation}\label{def:gradnodal2}
\nabla_\disc\bv_h = \nabla\bv_h 
+ ( \stiff^{-1/2}\stab^{1/2}-{\rm Id})(\nabla\bv_h-\Pi^*_h\nabla\bv_h)
=\nabla\bv_h + \mathcal L_h \nabla\bv_h
\end{equation}
where $\mathcal L_h=(\stiff^{-1/2}\stab^{1/2}-{\rm Id})(\mbox{Id}-\Pi^*_h):\L^2(\O)^d
\to \L^2(\O)^d$ has a norm bounded independently on $h$ and converges pointwise
to $0$.

Let us first consider the consistency property. For any $\bvarphi\in \H^1_{\Gd}(\O)$,
by quasi-uniformity of the sequence of meshes, there exists $\bv_n\in \bV_{h_n}=\X{\disc_n}$
such that $\bv_n=\Pi_{\disc_n}\bv_n\to \bvarphi$ in $\L^2(\O)$ and
$\nabla \bv_n\to \nabla\bvarphi$ in $\L^2(\O)^d$.
We have
\[
||\mathcal L_{h_n}\nabla\bv_n||_{\L^2(\O)^d}\le
||\mathcal L_{h_n}||_{\L^2(\O)^d\to \L^2(\O)^d}||\nabla\bv_n-\nabla \varphi||_{\L^2(\O)^d}
+||\mathcal L_{h_n}\nabla\bvarphi||_{\L^2(\O)^d}
\]
and, by the properties of $\mathcal L_{h_n}$,
both terms in the right-hand side tend to $0$.
Combined with \eqref{def:gradnodal2} this
proves that $\nabla_{\disc_n}\bv_n\to \nabla\bvarphi$ in $\L^2(\O)^d$, which
concludes the proof of the consistency of $(\disc_n)_{n\in\N}$.

Coercivity follows from the following comparisons between $\nabla$, $\nabla_{\disc_n}$
and $\strain$, $\strain_{\disc_n}$: there exists $\ctel{cst1},\ctel{cst2}>0$ not depending on $n$ such that,
for any $\bv\in \bV_{h_n}=\X{\disc_n}$,
\begin{eqnarray}
\label{comp:grad}
&&\cter{cst1}||\nabla_{\disc_n}\bv||_{\L^2(\O)^d}\le ||\nabla\bv||_{\L^2(\O)^d}
\le \cter{cst2} ||\nabla_{\disc_n}\bv||_{\L^2(\O)^d}\,,\\
\label{comp:strain}
&&\cter{cst1}||\strain_{\disc_n}(\bv)||_{\L^2(\O)^d}\le ||\strain(\bv)||_{\L^2(\O)^d}
\le \cter{cst2} ||\strain_{\disc_n}(\bv)||_{\L^2(\O)^d}.
\end{eqnarray}
Indeed, with these two estimates, the coercivity of $(\disc_n)_{n\in\N}$ is
a straightforward consequence of the Poincar\'e, trace and K\"orn's
inequalities in $\H^1_{\Gd}(\O)$. 
Since the proofs of \eqref{comp:grad} and \eqref{comp:strain} are similar,
we only consider the first one. Using 
$||\Pi^*_{h_n}\nabla\bv||_{\L^2(\O)^d}\le ||\nabla\bv||_{\L^2(\O)^d}$,
\eqref{def:gradnodal2} immediately gives the first inequality in \eqref{comp:grad}.
To establish the second one, we just notice, applying $\Pi_{h_n}^*$ to \eqref{def:gradnodal}
that $\Pi_{h_n}^*\nabla_{\disc_n}\bv=\Pi_{h_n}^*\nabla\bv$, which gives, plugged
into \eqref{def:gradnodal},
\[
\nabla\bv = \Pi^*_{h_n}\nabla_{\disc_n}\bv
+ \stab^{-1/2}\stiff^{1/2}\big(\nabla_{\disc_n}\bv - \Pi^*_{h_n} \nabla_{\disc_n}\bv\big).
\]
The second estimate of \eqref{comp:grad} follows by taking the $\L^2(\O)^d$ norm
of this equality and using once more the fact that the orthogonal projection
$\Pi^*_{h_n}$ has norm $1$.

Limit-conformity is then easy to establish. For any $\bL\in \H_{\div}(\O,\Gn)^d$
and any $\bv\in \bV_{h_n}=\X{\disc_n}$, by using \eqref{def:gradnodal2} we have
\begin{eqnarray}
\lefteqn{\Bigg|
\int_\Omega \big(\nabla_{\disc_n} \bv(x):\bL(x) +  \Pi_{\disc_n} \bv(x) \cdot \div(\bL)(x)\big)  \d x
-\int_{\Gn}\gamma_\n(\bL)(x)\cdot\T_{\disc_n}(\bv)(x)\d S(x)\Bigg|}\nonumber\\
&\le &\Bigg|
\int_\Omega \big(\nabla \bv(x):\bL(x) +  \bv(x) \cdot \div(\bL)(x)\big)  \d x
-\int_{\Gn}\gamma_\n(\bL)(x)\cdot \gamma(\bv)(x)\d S(x)\Bigg|\nonumber\\
&&+\Bigg|\int_\Omega \mathcal L_{h_n}\nabla\bv(x):\bL(x)\d x\Bigg|=T_1+T_2.
\label{limconf:1}
\end{eqnarray}
By conformity of $\bV_{h_n}$ we have $T_1=0$. Thanks to \eqref{comp:grad}
and denoting by $\mathcal L_{h_n}^{\star}=({\rm Id}-\Pi_h^*)(\stab^{1/2}\stiff^{-1/2}-{\rm Id})$
the dual operator of $\mathcal L_{h_n}$, we can write
\begin{eqnarray*}
T_2&=&\Bigg|\int_\Omega \nabla\bv(x)
:\mathcal L_{h_n}^{\star}\bL(x)\d x\Bigg|\\
&\le& ||\nabla \bv||_{\L^2(\O)^d}||\mathcal L_{h_n}^{\star}\bL||_{\L^2(\O)^d}
\le \cter{cst2} ||\nabla_{\disc_n} \bv||_{\L^2(\O)^d}||\mathcal L_{h_n}^{\star}\bL||_{\L^2(\O)^d}.
\end{eqnarray*}
Plugged into \eqref{limconf:1}, this estimate on $T_2$ shows that
$W_{\disc_n}(\bL)\le \cter{cst2}||\mathcal L_{h_n}^{\star}\bL||_{\L^2(\O)^d}$.
As $\mathcal L_{h_n}^{\star}\to 0$ pointwise as $n\to\infty$,
this concludes the proof of the limit-conformity of $(\disc_n)_{n\in\N}$.

\begin{remark} Reference \cite{PS06} provides an $\mathcal O(h)$ error estimate for
\eqref{stab-nodal} under very strong assumptions
on the solution to the continuous equation \eqref{eq:elas-strong}, namely
$\buex\in C^2(\overline{\Omega})$. Embedding \eqref{stab-nodal} into the
Gradient Scheme framework allowed us to establish the same error estimate
under no regularity assumption on the exact solution (see Theorem \ref{thm:error-est})
and that, contrary to what is written in \cite[p848]{PS06}, the
smoothness of the solution is not required for the error analysis of the method.
\end{remark}

\begin{remark}\label{rem:nl} As a consequence of these properties and of Theorem \ref{th:convnl},
we deduce that the Gradient Scheme discretisation \eqref{def:gsnodal}--\eqref{def:gradnodal}
coming from the stabilised nodal strain formulation of the linear elasticity
equations can be used to define a ``stabilised nodal strain formulation for
\emph{non-linear} elasiticity'' \eqref{grad-scheme}, and gives a converging scheme
for these equations. In this case, the tensors $\stiff$ and $\mathbb{D}$ in
\eqref{def:gradnodal} should be chosen accordingly to the considered non-linear equation,
e.g. by selecting linear tensors with Lam\'e's coefficients of the correct order of magnitude
with respect to the non-linear model.
\end{remark}

\begin{remark} We can also construct the ``nodal stabilised''
Gradient Discretisation $\disc$ by \eqref{def:gsnodal}--\eqref{def:gradnodal}
starting from a non-conforming Finite Element discretisation $\bV_h$ (or,
for that matter, any initial Gradient Discretisation built on a polygonal
discretisation of $\O$ as defined in \cite{DRO13}). In this case,
the preceding reasoning shows that if $(\bV_{h_n})_{n\in\N}$ is consistent,
limit-conforming and coercive then the corresponding nodal stabilised Gradient Discretisation
$(\disc_n)_{n\in\N}$ is also consistent, limit-conforming and coercive.
\end{remark}

\subsection{Hu-Washizu-based formulation on quadrilateral meshes}

We now consider a Finite Element method based on a modified Hu-Washizu formulation \cite{LRW06} 
for quadrilateral meshes. We start with the statically 
condensed displacement-based formulation in \cite{LRW06} of 
the following form: find $\bu_h \in  \bV_h $ such that
\begin{equation} \label{eq:disab}
\int_\O P_{S_h} \strain(\bv_h)(x) : \stiff_h P_{S_h} \strain(\bu_h)(x)\d x = \ell(\bv_h),
\quad \bv_h \in  \bV_h \ ,
\end{equation}
where $\bV_h$ is the standard conforming Finite Element 
space constructed from piecewise bilinear polynomials on a reference element, 
$P_{S_h}$ is the $L^2$ orthogonal projection onto the 
discrete space of stress $S_h$, and $\stiff_h$ is some positive-definite 
symmetric operator approximating the classical linear elasticity tensor $\stiff$
with constant Lam\'e coefficients, $\stiff\btau=\lambda\tr(\btau)\mathbf{I}+2\mu\btau$.
We note that the space of stress $S_h\subset \L^2(\Omega)^d $ is defined element-wise, 
and there is no continuity condition for its element across the boundary of cell 
in $\CT_h$. Various Finite Element methods used in alleviating locking effects are 
derived using this formulation \cite{LRW06,DLRW06}. Among them, the most popular 
methods are the assumed enhanced strain method of Simo and Rifai \cite{SR90}, the 
strain gap method of Romano,  Marrotti de Sciarra and Diaco \cite{RMD01}, and 
the mixed enhanced strain method of Kasper and Taylor  \cite{KT00}. 
We now consider the action of the operator $\stiff_h $ on a tensor 
$\bd_h= P_{S_h}\strain(\bu_h)$ as derived in \cite{LRW06}.
We use an orthogonal decomposition of $S_h$
in the form 
\[
S_h=S_h^c \oplus S_h^t \ ,
\]
where
\[
S_h^c := \{ \btau \in S_h \ | \ \ \stiff \btau \in S_h \}
\]
and $S_h^t$ is the orthogonal complement of $S_h^c$.
We consider the case where  the ope\-rator $\stiff_h$ is expressed as \cite{LRW06}
\begin{equation}\label{def:Ch}
\stiff_h \bd_h = \stiff P_{S_h^c} \bd_h + \theta P_{S_h^t} \bd_h
\end{equation}
where $P_{S_h^c}$ and $P_{S_h^t}$ are the orthogonal projections
onto $S_h^c$ and $S_h^t$ and $\theta>0$ is a constant only depending upon the Lam\'e coefficients
$\lambda,\mu$ of $\stiff$ and upon the parameter $\alpha>0$ of the modified
three-field  Hu-Washizu formulation \cite{LRW06}.
When the modified Hu-Washizu formulation 
is equivalent to the Hellinger-Reissner formulation, $\theta$ does not depend on $\alpha$.

\begin{remark}
The expression for the action of $\stiff_h$ is obtained in \cite{LRW06} using Voigt notation for tensors. 
However, we give here the expression for the discrete space of stress using 
the full tensor notation so that  we have 
\[
P_{S_h} \strain(\bu_h) = 
\frac{1}{2} \left(P_{S_h} (\nabla \bu_h) + P_{S_h} (\nabla \bu_h)^T\right).
\]
\end{remark}

We restrict ourselves, for simplicity of presentation, to the two-dimensional case, where  $\bd_h$ is a $2 $ by $2$ tensor. 
We consider three choices for $S_h$, where this space is generated (through
conformal transformations) from bases $S_{\Box}$  defined on $\hat{K}:=(-1,1)^2$.
Let these three choices be denoted by $S_h^i$ and $S^i_{\Box}$, 
$1 \leq i \leq 3$.
\begin{eqnarray*}
S_{\Box}^1 := \left[\begin{array}{ccc}
\mbox{span}\{1,\hat y\} & \mbox{span}\{1\} \\
\mbox{span}\{1\}&\mbox{span}\{1, \hat x\} 
\end{array}\right], \quad 
S_{\Box}^2 := \left[\begin{array}{ccc}
\mbox{span}\{1,\hat y\} &\mbox{span}\{1,\hat x,\hat y \}\\
\mbox{span}\{1,\hat x,\hat y\}&\mbox{span}\{1, \hat x \}
\end{array}\right],
 \end{eqnarray*}
 and 
\begin{eqnarray*}
S_{\Box}^3 := \left[\begin{array}{ccc}
\mbox{span}\{1\} &\mbox{span} \{1, \hat x,\hat y\}\\
\mbox{span}\{1, \hat x,\hat y\} &\mbox{span}\{1\} 
\end{array}\right]
\end{eqnarray*}
While the spherical part of the stress might be polluted by
checkerboard modes  as in the case of the $Q_1-P_0$ element,
it is proved that the error in displacement satisfies a 
$\lambda$-independent \emph{a priori} error estimate \cite{LRW06}. 

\medskip

Let us now prove that if $S_h=S_h^i$ for some $1\le i\le 3$
then \eqref{eq:disab} is a Gradient Scheme. We
define 
\begin{equation}\label{huw:grad}
\begin{array}{c}
\dsp \X{\disc}=\bV_h\,,\quad \Pi_\disc\bv_h=\bv_h\,,\quad \T_\disc\bv_h=\gamma(\bv_h)_{|\Gn}
\mbox{ and }\\
\dsp\nabla_\disc \bv_h = P_{S_h^c}\nabla\bv_h + \sqrt{\theta}\,\stiff^{-1/2} P_{S_h^t}\nabla \bv_h.
\end{array}\end{equation}
We note that, by symmetry of $\stiff$, $S_h^c$ and $S_h^t$ are closed under transposition and
therefore the projections onto those spaces commute with the transposition. By
Lemma \ref{lem:root}, the definition of $\nabla_\disc$ thus shows that 
\begin{equation}\label{huw:strain}
\strain_\disc (\bv_h) = P_{S_h^c}\strain(\bv_h) + \sqrt{\theta}\,\stiff^{-1/2} P_{S_h^t}\strain(\bv_h).
\end{equation}

We now prove that the Gradient Scheme corresponding to the Gradient Discretisation
$\disc=(\X{\disc},\Pi_\disc,\T_\disc,\nabla_\disc)$
is precisely the Hu-Washizu scheme \eqref{eq:disab}. Let us first start with a lemma.

\begin{lemma}\label{lem:stab}
For any of the choices $S^i_h$ ($1\le i\le 3$) described above and for any linear
elasticity tensor $\stab$, $(S^i_h)^c$ is closed under $\stab$,
that is $\stab \bL\in (S^i_h)^c$ whenever $\bL \in (S^i_h)^c$.
In particular,
\begin{equation}\label{eq:orth}
\forall \bL,\bG\in \L^2(\O)^d\,,\quad
\int_\O \mathbb{D}P_{(S^i_h)^c}\bL(x):P_{(S^i_h)^t}\bG(x)\d x=0.
\end{equation}
\end{lemma}

\begin{proof} If $\btau\in (S_h^i)^c$ then $\tr(\btau)\mathbf{I}
=\lambda^{-1}(\stiff\btau-2\mu\btau)\in S_h^i$. The definitions of $S^i_h$ then shows,
by examining the coefficients $(1,1)$ and $(2,2)$ of $\tr(\btau)\mathbf{I}$, that
$\tr(\btau)\in\mbox{span}\{1,\hat{y}\}\cap\mbox{span}\{1,\hat{x}\}=\mbox{span}\{1\}$
and thus that $\tr(\btau)$ is constant.

By Lemma \ref{lem:comp}, we see that $\stiff\stab$ is a linear
elasticity tensor with some Lam\'e coefficients $(\alpha,\beta)$ and therefore
$\stiff\stab\btau =\alpha\tr(\btau)\mathbf{I} + 2\beta\btau$.
The second term in this right-hand side clearly belongs to $S_h^i$
and, since $\tr(\btau)$ is constant, it is equally obvious that the first term in the
right-hand side belongs to $S_h^i$ (which contains $\mbox{span}\{\mathbf{I}\}$).
Hence, $\stab\btau\in (S_h^i)^c$ whenever $\btau\in (S_h^i)^c$.
Formula \eqref{eq:orth} is a consequence of this and of the orthogonality
of $S_h^c$ and $S_h^t$. \end{proof}

~

We now consider the left-hand side of \eqref{grad-scheme:lin}. Using
\eqref{eq:orth} with $\mathbb{D}=\stiff^{1/2}$ (which is a linear
elasticity tensor by Lemma \ref{lem:comp}), the cross-products involving $\stiff^{1/2}P_{S_h^c}$
and $P_{S_h^t}$
which appear when plugging \eqref{huw:strain} into \eqref{grad-scheme:lin} vanish and we obtain
\begin{multline}\label{huw:lhs1}
\int_\O \stiff(x)\strain_\disc(\bu_h)(x):\strain_\disc(\bv_h)(x)\d x
\\
= \int_\O \stiff(x)P_{S_h^c}\strain(\bu_h)(x): P_{S_h^c}\strain(\bu_h)(x)\d x
+\int_\O \theta P_{S_h^t}\strain(\bu_h)(x): P_{S_h^t}\strain(\bu_h)(x)\d x.
\end{multline}
Using now the definition \eqref{def:Ch} of $\stiff_h$ and the orthogonality
property \eqref{eq:orth} with $\mathbb{D}=\stiff$,
the left-hand side of \eqref{eq:disab} can be written
\begin{multline} \label{huw:lhs2}
\int_\O \stiff_h P_{S_h} \strain(\bu_h)(x):P_{S_h} \strain(\bv_h)(x) \d x \\
=
\int_\O \left[\stiff(x) P_{S_h^c} \strain(\bu_h)(x)+\theta P_{S_h^t}\strain(\bu_h)(x)\right]
:\left[P_{S_h^c} \strain(\bv_h)(x) +P_{S_h^t} \strain(\bv_h)(x)\right]\d x\\
= \int_\O \stiff(x) P_{S_h^c} \strain(\bu_h)(x):P_{S_h^c} \strain(\bv_h)(x)\d x
+\int_\O\theta P_{S_h^t}\strain(\bu_h)(x):P_{S_h^t} \strain(\bv_h)(x)\d x.
\end{multline}
Equations \eqref{huw:lhs1} and \eqref{huw:lhs2} show that the left-hand sides
of the Gradient Scheme \eqref{grad-scheme:lin} and of the Hu-Washizu formulation
\eqref{eq:disab} are identical. As the right-hand sides of these equations
are trivially identical (by definition of $\Pi_\disc$ and $\T_\disc$),
this shows that the statically condensed  Hu-Washizu  formulation \cite{LRW06} is the
Gradient Scheme corresponding to the Gradient Discretisation defined by \eqref{huw:grad}.

\medskip

Let us now see that the Gradient Discretisation \eqref{huw:grad} satisfies
the properties defined in Section \ref{sec:gs}. 
The coercivity is again a consequence of \eqref{comp:grad} and \eqref{comp:strain}
that we can prove in the following way. First, since the norms of $P_{S_h^c}$
and $P_{S_h^t}$ are bounded by $1$, the definition \eqref{huw:grad} of $\nabla_\disc$
and the property \eqref{huw:strain} of $\strain_\disc$ immediately give
the first inequalities in \eqref{comp:grad} and \eqref{comp:strain}.
We then write, from \eqref{huw:strain},
\begin{equation}\label{huw:coer}
\stiff^{1/2}\strain_\disc(\bv_h)= \stiff^{1/2}P_{S_h^c}\strain(\bv_h)
+\sqrt{\theta}\, P_{S_h^t}\strain(\bv_h).
\end{equation}
By Lemmas \ref{lem:stab} and \ref{lem:comp}, we have $\stiff^{1/2}P_{S_h^c}\nabla\bv_h\in S_h^c$
and \eqref{huw:coer} thus shows that
$P_{S_h^c}\stiff^{1/2}\strain_\disc(\bv_h)=\stiff^{1/2}P_{S_h^c}\strain(\bv_h)$
and $P_{S_h^t}\stiff^{1/2}\strain_\disc(\bv_h)
=\sqrt{\theta}\, P_{S_h^t}\strain(\bv_h)$.
This allows us to write
\begin{eqnarray*}
P_{S_h}\strain(\bv_h)&=& P_{S_h^c}\strain(\bv_h)+P_{S_h^t}\strain(\bv_h)\\
&=&\stiff^{-1/2}P_{S_h^c}\stiff^{1/2}\strain_\disc(\bv_h)
+\sqrt{\theta}^{\;-1}P_{S_h^t}\stiff^{1/2}\strain_\disc(\bv_h).
\end{eqnarray*}
This relation shows that $||P_{S_h}\strain(\bv_h)||_{\L^2(\O)^d}\le 
\ctel{huw:cst}||\strain_{\disc}(\bv_h)||_{\L^2(\O)^d}$ with $\cter{huw:cst}$ not depending on $h$
or $\bv_h$. Since it can be proved (see \cite{LRW06}) that $||\strain(\bv_h)||_{\L^2(\O)^d}
\le \ctel{huw:cst2}||P_{S_h}\strain(\bv_h)||_{\L^2(\O)^d}$ with $\cter{huw:cst2}$
not depending on $h$ or $\bv_h$, the second inequality in \eqref{comp:strain}
follows immediately. The second inequality in \eqref{comp:grad} can then be established
by using the continuous K\"orn inequality $||\nabla\bv_h||_{\L^2(\O)^d}\le
\ctel{huw:cst3}||\strain(\bv_h)||_{\L^2(\O)^d}$ and the
second inequality of \eqref{comp:strain} that we just established.

To establish the consistency and limit-conformity of the Gradient Discretisation,
we notice that 
\begin{equation}\label{huw:grad2}
\nabla_\disc \bv_h = \nabla\bv_h + (P_{S_h^c}-\mbox{Id})\nabla\bv_h
+\sqrt{\theta}\,\stiff^{-1/2}P_{S_h^t}\nabla \bv_h = \nabla\bv_h + \mathcal L_h\nabla\bv_h
\end{equation}
where $\mathcal L_h=P_{S_h^c}-\mbox{Id}+\sqrt{\theta}\,\stiff^{-1/2}P_{S_h^t}:
\L^2(\O)^d\to \L^2(\O)^d$ is a self-adjoint operator (because $\sqrt{\theta}\,\stiff^{-1/2}$
is constant) whose norm is bounded independently on $h$.
As $S_h^c$ always contains the set of constant tensors $S_h^0$ and $P_{S_h^0}\to \mbox{Id}$
as $h\to 0$, we have $P_{S_h^c}=P_{S_h^c}(\mbox{Id}-P_{S_h^0})+P_{S_h^0}\to \mbox{Id}$
and $P_{S_h^t}=P_{S_h^t}(\mbox{Id}-P_{S_h^0})\to 0$ pointwise as $h\to 0$.
Hence, $\mathcal L_h\to 0$ pointwise as $h\to 0$.
Expression \eqref{huw:grad2} then allows us to prove
the consistency and limit-conformity of the Gradient Discretisation \eqref{huw:grad} by using the same
techniques as in Section \ref{sec:propstabnodal}.

\begin{remark} The same construction can be made when $\stiff$ is only
piecewise constant on $\CT_h$. 
\end{remark}

\begin{remark}
In contrast to \cite{BCR04,LRW06}, the convergence result of Theorem \ref{thm:error-est}
is obtained for the Hu-Washizu scheme without assuming the full $H^2$-regularity of the solution. 
Moreover, as in Remark \ref{rem:nl}, this construction also gives
a converging Hu-Washizu-based scheme for non-linear elasticity equations.
\end{remark}


%

\subsection{Technical lemmas}

\begin{lemma}\label{lem:comp}
If $\stiff_1$ and $\stiff_2$ are linear elasticity tensors in $\R^d$ with Lam\'e
coefficients $(\lambda_1,\mu_1)$ and $(\lambda_2,\mu_2)$, then, for
any $\bL\in \R^{d\times d}$,
\begin{equation}\label{form:comp}
\stiff_1\stiff_2\bL = (\lambda_1\lambda_2 d +2\mu_1\lambda_2+2\mu_2\lambda_1)
\tr(\bL)\mathbf{I}+4\mu_1\mu_2 \bL.
\end{equation}
If $\stiff$ is a linear elasticity tensor with Lam\'e coefficients $(\lambda,\mu)$,
then
\begin{equation}\label{form:root}
\stiff^{1/2}\bL = \frac{\sqrt{2\mu+\lambda d}-\sqrt{2\mu}}{d}\tr(\bL)\mathbf{I}
+\sqrt{2\mu}\bL.
\end{equation}
\end{lemma}

\begin{proof} Formula \eqref{form:comp} is obtained by straightforward computation,
and Formula \eqref{form:root} by looking for $\stiff^{1/2}$ as a linear elasticity
tensor with coefficients $(\alpha,\beta)$ such that $\stiff^{1/2}\stiff^{1/2}=\stiff$,
which boils down from \eqref{form:comp} to solving $\alpha^2d+4\alpha\beta=\lambda$
and $4\beta^2=2\mu$.
\end{proof}

\begin{lemma}\label{lem:root} If $\mathbb{E}:(\R^{d\times d},:)\to (\R^{d\times d},:)$
is symmetric positive definite and satisfies, for all $\bL\in \R^{d\times d}$,
$(\mathbb{E}\bL)^T=\mathbb{E}\bL^T$, then $\mathbb{E}^{1/2}$ also satisfies
this property.
\end{lemma}

\begin{proof} Let $\mathcal L:\R^{d\times d}\to \R^{d\times d}$
be the endomorphism $\mathcal L \bL=(\mathbb{E}^{1/2}\bL^T)^T$.
Using $\bL:\bG=\bL^T:\bG^T$ and the symmetric positive definite character
of $\mathbb{E}^{1/2}$, it is easy to check that
$\mathcal L$ is symmetric positive definite. Moreover, by assumption on $\mathbb{E}$,
$\mathcal L^2\bL=(\mathbb{E}^{1/2}[(\mathbb{E}^{1/2}\bL^T)^T]^T)^T
=(\mathbb{E}^{1/2}\mathbb{E}^{1/2}\bL^T)^T=(\mathbb{E}\bL^T)^T=
\mathbb{E}\bL$. Henceforth, $\mathcal L$ is the
symmetric positive definite square root $\mathbb{E}^{1/2}$ of $\mathbb{E}$
and thus $\mathbb{E}^{1/2}\bL^T = \mathcal L(\bL^T)=(\mathbb{E}^{1/2}\bL)^T$,
which completes the proof.
\end{proof}

\section{Conclusion}

In this work, we developed the Gradient Scheme framework for linear
and non-linear elasticity equations. We proved that this framework
makes possible error estimates (for linear equations) and convergence
analysis (for non-linear equations) of numerical methods under very
few assumptions. In particular, these results hold without assuming the 
full $H^2$-regularity of the exact solution, which can be lost in the cases
of composite materials or strongly non-linear models.

We showed that many classical and modern numerical schemes
developed in the literature for elasticity equations are actually
Gradient Schemes. We even established that some three-field schemes,
based on a modified Hu-Washizu formulation and designed to be stable
in the quasi-incompressible limit, are also Gradient Schemes after
being recast in a displacement-only formulation by static condensation.

Since Gradient Schemes are seamlessly applicable to both linear and non-linear
equations, the embedding into this framework of numerical methods solely developed for
linear elasticity also allowed us to show how to adapt those methods
to non-linear elasticity, while retaining nice stability and convergence
properties.

\bibliographystyle{plain}
\bibliography{gs-elasticity-submitted}{}

\begin{thebibliography}{10}

\bibitem{AEL13}
L.~Ag\'elas, R.~Eymard, and S~Lemaire.
\newblock A locking-free euler-gradient scheme approximation of biot’s
  consolidation problem on general meshes.
\newblock In preparation.

\bibitem{BAR02}
M.~A. Barrientos, G.~N. Gatica, and E.~P. Stephan.
\newblock A mixed finite element method for nonlinear elasticity: two-fold
  saddle point approach and a-posteriori error estimate.
\newblock {\em Numer. Math.}, 91(2):197--222, 2002.

\bibitem{Bra01}
D.~Braess.
\newblock {\em Finite Elements. Theory, fast solver, and applications in solid
  mechanics}.
\newblock Cambridge University Press, Second Edition, 2001.

\bibitem{BCR04}
D.~Braess, C.~Carstensen, and B.D. Reddy.
\newblock Uniform convergence and a posteriori error estimators for the
  enhanced strain finite element method.
\newblock {\em Numerische Mathematik}, 96:461--479, 2004.

\bibitem{BM05}
D.~Braess and P.-B. Ming.
\newblock A finite element method for nearly incompressible elasticity
  problems.
\newblock {\em Mathematics of Computation}, 74:25--52, 2005.

\bibitem{BS94}
S.C. Brenner and L.R. Scott.
\newblock {\em The Mathematical Theory of Finite Element Methods}.
\newblock Springer--Verlag, New York, 1994.

\bibitem{BS92}
S.C. Brenner and L.~Sung.
\newblock Linear finite element methods for planar linear elasticity.
\newblock {\em Mathematics of Computation}, 59:321--338, 1992.

\bibitem{BH06}
E.~Burman and P.~Hansbo.
\newblock A stabilized non-conforming finite element method for incompressible
  flow.
\newblock {\em Computer Methods in Applied Mechanics and Engineering},
  195:2881--2899, 2006.

\bibitem{CD04}
C.~Carstensen and G.~Dolzmann.
\newblock An a priori error estimate for finite element discretizations in
  nonlinear elasticity for polyconvex materials under small loads.
\newblock {\em Numerische Mathematik}, 97:67--80, 2004.

\bibitem{CCC10b}
M.~Cervera, M.~Chiumenti, and R.~Codina.
\newblock Mixed stabilized finite element methods in nonlinear solid mechanics
  {P}art {II}: strain localization.
\newblock {\em Comput. Methods Appl. Mech. Engrg.}, 199(37-40):2571--2589,
  2010.

\bibitem{Cia78}
P.G Ciarlet.
\newblock {\em The Finite Element Method for Elliptic Problems}.
\newblock North Holland, Amsterdam, 1978.

\bibitem{Cia88}
P.G. Ciarlet.
\newblock {\em Mathematical Elasticity Volume I: Three-Dimensional Elasticity}.
\newblock North-Holland, Amsterdam, 1988.

\bibitem{DEI85}
K.~Deimling.
\newblock {\em Nonlinear functional analysis}.
\newblock Springer-Verlag, Berlin, 1985.

\bibitem{DLRW06}
J.K. Djoko, B.P. Lamichhane, B.D. Reddy, and B.I. Wohlmuth.
\newblock Conditions for equivalence between the {H}u-{W}ashizu and related
  formulations, and computational behavior in the incompressible limit.
\newblock {\em Computer Methods in Applied Mechanics and Engineering},
  195:4161--4178, 2006.

\bibitem{DRO06}
J.~Droniou.
\newblock Finite volume schemes for fully non-linear elliptic equations in
  divergence form.
\newblock {\em M2AN Math. Model. Numer. Anal.}, 40(6):1069--1100 (2007), 2006.

\bibitem{DRO13}
J.~Droniou, R.~Eymard, T.~Gallou\"et, C.~Guichard, and R~Herbin.
\newblock {\em Gradient schemes for elliptic and parabolic problems}.

\bibitem{DRO12}
J.~Droniou, R.~Eymard, T.~Gallou{\"e}t, and R.~Herbin.
\newblock Gradient schemes: a generic framework for the discretisation of
  linear, nonlinear and nonlocal elliptic and parabolic equations.
\newblock {\em Math. Models Methods Appl. Sci.}, 2012.
\newblock To appear.

\bibitem{EYM13-2}
R.~Eymard, P.~F\'eron, T.~Gallou\"et, R.~Herbin, and C.~Guichard.
\newblock Gradient schemes for the stefan problem.
\newblock 2013.
\newblock submitted.

\bibitem{EYM09}
R.~Eymard, T.~Gallou{\"e}t, and R.~Herbin.
\newblock Cell centred discretisation of non linear elliptic problems on
  general multidimensional polyhedral grids.
\newblock {\em J. Numer. Math.}, 17(3):173--193, 2009.

\bibitem{EYM12}
R.~Eymard, C.~Guichard, and R.~Herbin.
\newblock Small-stencil 3{D} schemes for diffusive flows in porous media.
\newblock {\em ESAIM Math. Model. Numer. Anal.}, 46(2):265--290, 2012.

\bibitem{EYM11-2}
R.~Eymard, A.~Handlovi{\v{c}}ov{\'a}, R.~Herbin, K.~Mikula, and
  O.~Sta{\v{s}}ov{\'a}.
\newblock Gradient schemes for image processing.
\newblock In {\em Finite volumes for complex applications. {VI}. {P}roblems \&
  perspectives. {V}olume 1, 2}, volume~4 of {\em Springer Proc. Math.}, pages
  429--437. Springer, Heidelberg, 2011.

\bibitem{EYM11}
R.~Eymard and R.~Herbin.
\newblock Gradient scheme approximations for diffusion problems.
\newblock In {\em Finite volumes for complex applications. {VI}. {P}roblems \&
  perspectives. {V}olume 1, 2}, volume~4 of {\em Springer Proc. Math.}, pages
  439--447. Springer, Heidelberg, 2011.

\bibitem{EYM13}
R.~Eymard and R.~Herbin.
\newblock Mixed finite element methods and gradient schemes for underground
  flow simulations.
\newblock In {\em Proc. of the 5th International Con- ference on Approximation
  Methods and Numerical Modelling in Environment and Nat- ural Resources},
  Granada, Spain, 2013.
\newblock submitted.

\bibitem{FAL90}
R.~S. Falk and M.~E. Morley.
\newblock Equivalence of finite element methods for problems in elasticity.
\newblock {\em SIAM J. Numer. Anal.}, 27:1486--1505, 1990.

\bibitem{FB81}
D.P. Flanagan and T.~Belytschko.
\newblock A uniform strain hexahedron and quadrilateral with orthogonal
  hourglass control.
\newblock {\em International Journal for Numerical Methods in Engineering},
  17:679--706, 1981.

\bibitem{GS02}
G.N. Gatica and E.P. Stephan.
\newblock A mixed-{FEM} formulation for nonlinear incompressible elasticity in
  the plane.
\newblock {\em Numerical Methods for Partial Differential Equations},
  18:105--128, 2002.

\bibitem{KT00}
E.~P. Kasper and R.~L. Taylor.
\newblock A mixed-enhanced strain method. {P}art {I}: geometrically linear
  problems.
\newblock {\em Computers and Structures}, 75:237--250, 2000.

\bibitem{KNO00}
P.~Knobloch.
\newblock On korn's inequality for nonconforming finite elements.
\newblock Technical report, 2000.
\newblock Band 20, Heft 3.

\bibitem{KMR01}
V.A. Kozlov, V.G. Maz'ya, and J.~Rossmann.
\newblock {\em Spectral Problems Associated with Corner Singularities of
  Solutions to Elliptic Equations}.
\newblock Mathematical Surveys and Monographs 85. American Mathematical
  Society, Providence, RI, 2001.

\bibitem{BIS09}
B.~P. Lamichhane.
\newblock Mortar finite elements for coupling compressible and nearly
  incompressible materials in elasticity.
\newblock {\em Int. J. Numer. Anal. Model.}, 6(2):177--192, 2009.

\bibitem{Lam09p}
B.P. Lamichhane.
\newblock From the {H}u--{W}ashizu formulation to the average nodal strain
  formulation.
\newblock {\em Computer Methods in Applied Mechanics and Engineering},
  198:3957--3961, 2009.

\bibitem{LRW06}
B.P. Lamichhane, B.D. Reddy, and B.I. Wohlmuth.
\newblock Convergence in the incompressible limit of finite element
  approximations based on the {H}u-{W}ashizu formulation.
\newblock {\em Numerische Mathematik}, 104:151--175, 2006.

\bibitem{LER65}
J.~Leray and J.-L. Lions.
\newblock Quelques r\'esultats de {V}i\v sik sur les probl\`emes elliptiques
  nonlin\'eaires par les m\'ethodes de {M}inty-{B}rowder.
\newblock {\em Bull. Soc. Math. France}, 93:97--107, 1965.

\bibitem{MIN63}
G.J. Minty.
\newblock On a ``monotonicity'' method for the solution of non-linear equations
  in {B}anach spaces.
\newblock {\em Proceedings of the National Academy of Sciences of the United
  States of America}, 50(6):1038, 1963.

\bibitem{NEC86}
J.~Ne{\v{c}}as.
\newblock {\em Introduction to the theory of nonlinear elliptic equations}.
\newblock A Wiley-Interscience Publication. John Wiley \& Sons Ltd.,
  Chichester, 1986.
\newblock Reprint of the 1983 edition.

\bibitem{PS06}
M.~A. Puso and J.~Solberg.
\newblock A stabilized nodally integrated tetrahedral.
\newblock {\em International Journal for Numerical Methods in Engineering},
  67:841--867, 2006.

\bibitem{QV94}
A.~Quarteroni and A.~Valli.
\newblock {\em Numerical approximation of partial differential equations}.
\newblock Springer--Verlag, Berlin, 1994.

\bibitem{RMD01}
G.~Romano, F.~Marrotti~de Sciarra, and M.~Diaco.
\newblock Well-posedness and numerical performances of the strain gap method.
\newblock {\em Int. J. Numer. Meth. Engrg.}, 51:103--126, 2001.

\bibitem{SR90}
J.C. Simo and M.S. Rifai.
\newblock A class of assumed strain method and the methods of incompatible
  modes.
\newblock {\em Int. J. Numer. Meths. Engrg.}, 29:1595--1638, 1990.

\end{thebibliography}

\end{document}